\PassOptionsToPackage{lowtilde}{url}
\documentclass[onefignum,onetabnum]{siamart171218}

\newboolean{ispreprint}
\setboolean{ispreprint}{false}
\setboolean{ispreprint}{true}
\usepackage{Local_definitions}

\ifpdf
\hypersetup{
	pdftitle={Intrinsic formulation of KKT conditions and CQs on smooth manifolds},
	pdfauthor={R.~Bergmann, and R.~Herzog}
}
\fi%

\begin{document}

\maketitle

\begin{abstract}
	Karush-Kuhn-Tucker (KKT) conditions for equality and inequality constrained
	optimization problems on smooth manifolds are formulated. Under the Guignard
	constraint qualification, local minimizers are shown to admit Lagrange
	multipliers. The linear independence, Mangasarian--Fromovitz, and Abadie
	constraint qualifications are also formulated, and the chain \eqq{LICQ implies
	MFCQ implies ACQ implies GCQ} is proved. Moreover, classical connections
	between these constraint qualifications and the set of Lagrange multipliers are
	established, which parallel the results in Euclidean space.
	The constrained Riemannian center of mass on the sphere serves as an illustrating numerical example.
\end{abstract}



\section{Introduction}
\label{sec:Introduction}

We consider constrained, nonlinear optimization problems
\begin{equation}
	\label{eq:nlp}
	\left\{
		\quad
		\begin{aligned}
			\text{Minimize} \quad & f(\bp) \text{ w.r.t.\ } \bp \in \MM, \\
			\text{s.t.} \quad & g(\bp) \le 0, \\
			\text{and} \quad & h(\bp) = 0,
		\end{aligned}
	\right.
\end{equation}
where $\MM$ is a smooth manifold.
The objective $f\colon\MM \to \R$ and the constraint functions $g\colon \MM \to
\R^m$ and $h\colon\MM \to \R^q$ are assumed to be functions of class $C^1$. The
main contribution of this paper is the development of first-order necessary
optimality conditions in Karush-Kuhn-Tucker (KKT) form,
well known
when $\MM =
\R^n$, under appropriate constraint qualifications (CQs). Specifically, we
introduce and discuss analogues of the linear independence,
Mangasarian--Fromovitz, Abadie and Guignard CQ, abbreviated as LICQ, MFCQ, ACQ
and GCQ, respectively; see for instance \cite{Solodov2010}, \cite{Peterson1973}
or~\cite[Ch.~5]{BazaraaSheraliShetty2006}.

It is well known that KKT conditions are of paramount importance in nonlinear
programming, both for theory and numerical algorithms. We refer the reader to
\cite{Kjeldsen2000} for an account of the history of KKT condition in the
Euclidean setting $\MM = \R^n$. A variety of programming problems in numerous
applications, however, are naturally given in a manifold setting. Well-known
examples for smooth manifolds include spheres, tori, the general linear group
$\GL{n}$ of non-singular matrices, the group of special orthogonal (rotation)
matrices $\SO{n}$, the Grassmannian manifold of $k$-dimensional subspaces of a
given vector space, and the orthogonal Stiefel manifold of orthonormal
rectangular matrices of a certain size. We refer the reader to
\cite{AbsilMahonySepulchre2008} for an overview and specific examples.
Recently optimization on manifolds
has gained interest, e.g., in image processing, where methods like the cyclic
proximal point algorithm by~\cite{Bacak2014}, half-quadratic
minimization by~\cite{BergmannChanHielscherPerschSteidl2016}, and the parallel
Douglas-Rachford algorithm by~\cite{BergmannPerschSteidl2016} have been
introduced. They were then applied to variational models from imaging, i.e.,
optimization problems of the form \eqref{eq:nlp}, where the manifold is given
by the power manifold $\MM^N$ 
with 
$N$ being the number of data items or
pixels.
We emphasize that all of the above consider \emph{unconstrained} problems on manifolds.

In principle, inequality and equality constraints in \eqref{eq:nlp} might be
taken care of by considering a suitable submanifold of $\MM$ (with boundary).
This is much like in the case $\MM = \R^n$, where one may choose not to include
some of the constraints in the Lagrangian but rather treat them as abstract
constraints. Often, however, there may be good reasons to consider constraints
explicitly, one of them being that Lagrange multipliers carry sensitivity
information for the optimal value function, although this is not addressed in
the present paper.

To the best of our knowledge, a systematic discussion of constraint
qualifications and KKT conditions for~\eqref{eq:nlp} is not available in the
literature. We are aware of~\cite{Udriste1988} where KKT conditions are derived
for convex inequality constrained problems and under a Slater constraint
qualification on a complete Riemannian manifold. 
To be precise, the objective is convex along geodesics, and the feasible set is described by a finite collection of inequality constraints which are likewise geodesically convex.
The work closest to ours
is~\cite{YangZhangSong2014}, where KKT and also second-order optimality conditions
are derived for~\eqref{eq:nlp} in the setting of a smooth Riemannian manifold
and under the assumption of LICQ. Other constraint qualifications are not
considered. 
The emphasis of the present paper is on constraint qualifications and first-order necessary conditions of KKT type, but in contrast to \cite{YangZhangSong2014} we do not discuss second-order optimality conditions.
We also mention~\cite{LedyaevZhu2007} where a framework for
generalized derivatives of non-smooth functions on smooth Riemannian manifolds
is developed and Fritz--John type optimality conditions are derived as an
application.
Recently, a discussion of algorithms for equality and inequality constrained problems
on Riemannian manifolds was performed in~\cite{LiuBoumal2019_preprint}

The novelty of the present paper is the formulation of analogues for a range of
constraint qualifications (LICQ, MFCQ, ACQ, and GCQ) in the smooth manifold
setting. We establish the classical \eqq{LICQ implies MFCQ implies ACQ implies
GCQ} and prove that KKT conditions are necessary optimality conditions under
any of these CQs. We also show that the classical connections between these
constraint qualifications and the set of Lagrange multipliers continue to hold,
e.g., Lagrange multipliers are generically unique if and only if LICQ holds.
Finally, our work shows that the smooth structure on a manifold is a framework
sufficient for the purpose of first-order optimality conditions. In particular,
we do not need to introduce a Riemannian metric as in~\cite{YangZhangSong2014}.

We wish to point out that optimality conditions can also be derived
by considering $\MM$ to be embedded in a suitable ambient Euclidean space
$\R^N$. This approach requires, however, to formulate additional, nonlinear
constraints in order to ensure that only points in $\MM$ are considered
feasible. Another drawback of such an approach is that the number of variables
grows since $N$ is larger than the manifold dimension.
In contrast to the embedding approach, we formulate KKT conditions and
appropriate constraint qualifications (CQs) using \emph{intrinsic} concepts on
the manifold $\MM$. This requires, in particular, the generalization of the
notions of tangent and linearizing cones to the smooth manifold setting.
The intrinsic point of view is also the basis of many optimization approaches
for problems on manifolds; see for
instance~\cite{AbsilMahonySepulchre2008,AbsilBakerGallivan2007,Boumal2015}.

We also mention that since CQs and KKT conditions are local concepts, the results of tis paper can be stated and derived in a different way:
	one can transcribe \eqref{eq:nlp} locally into an optimization problem in Euclidean space  and subsequently apply the theory of CQs and KKT in $\R^n$.
	This leads to equivalent definitions and results.
	However we find it more instructive to formulate CQs and KKT conditions using the language of differential geometry and to minimize the explicit use of charts.

The material is organized as follows. In \cref{sec:Background_material} we
review the necessary background material on smooth manifolds. Our main results
are given in~\cref{sec:KKT_and_CQs}, where KKT conditions are formulated and
shown to hold for local minimizers under the Guignard constraint qualifications.
We also formulate further constraint qualifications (CQs) and establish
\eqq{LICQ implies MFCQ implies ACQ implies GCQ}.
\Cref{sec:CQs_and_Lagrange_multipliers} is devoted to the connections between
CQs and the set of Lagrange multipliers. In \cref{sec:Example} we present an
application of the theory.

\subsection*{Notation}

Throughout the paper, $\varepsilon$ is a positive number whose value may vary
from occasion to occasion. We distinguish between column vectors (elements of
$\R^n$) and row vectors (elements of $\R_n$).
Moreover, we recall that a subset $K$ of a vector space $V$ is said to be a \emph{cone} if $\alpha K \subseteq K$ for all $\alpha > 0$. A cone $K$ may or may not be convex.

\section{Background Material}
\label{sec:Background_material}

In this section we review the required background material on smooth manifolds.
We refer the reader to \cite{Spivak1979,Aubin2001,Lee2003,Tu2011,Jost2017} for a
thorough introduction.

\begin{definition}
	\label{definition:smooth_manifold}
	Suppose that $\MM$ is a Hausdorff, second-countable topological space~$\MM$.
	One says that $\MM$ can be endowed with a smooth structure of dimension $n \in \N$ 
	if there exists an arbitrary index set $A$, a collection of open subsets
	$\{U_\alpha\}_{\alpha \in A}$ covering $\MM$, together with a collection of
	homeomorphisms (continuous functions with continuous inverses)
	$\varphi_\alpha\colon U_\alpha \to \varphi_\alpha(U_\alpha) \subseteq \R^n$,
	such that the transition maps~$\varphi_\beta \circ
	\varphi_\alpha^{-1}\colon\varphi_\alpha(U_\alpha \cap U_\beta) \to
	\varphi_\beta(U_\alpha \cap U_\beta)$ are of class $C^\infty$ for all
	$\alpha, \beta \in A$. A pair $(U_\alpha,\varphi_\alpha)$ is called a smooth
	chart, and the collection $\AA \coloneqq\{(U_\alpha,\varphi_\alpha)\}_{\alpha \in A}$ is a
	smooth atlas.
	Then the pair $(\MM,\AA)$ is called a smooth manifold.
\end{definition}
Well-known examples of smooth manifolds include $\R^n$, spheres, tori, $\GL{n}$,
$\SO{n}$, the Grassmannian manifold of $k$-dimensional subspaces of a given
vector space, and the orthogonal Stiefel manifold of orthonormal rectangular
matrices of a certain size; see for instance~\cite{AbsilMahonySepulchre2008}.
From now on, a smooth manifold $\MM$ will always be equipped with a given smooth
atlas $\AA$. In particular, $\R^n$ will be equipped with the standard atlas consisting
of the single chart $(\R^n,\id)$. Points on $\MM$ will be denoted by bold-face
letters such as $\bp$ and $\bq$.

Notions beyond continuity are defined by means of charts. In particular, the
assumed $C^1$-property of the objective $f\colon\MM \to \R$ means that $f \circ
\varphi_\alpha^{-1}$, defined on the open subset $\varphi_\alpha(U_\alpha)
\subseteq \R^n$ and mapping into $\R$, is of class $C^1$ for every chart
$(U_\alpha,\varphi_\alpha)$ from the smooth atlas. The $C^1$-property of the
constraint functions $g$ and $h$ is defined in the same way. Similarly, one may
speak of $C^1$-functions which are defined only in an open subset $U \subset
\MM$, by replacing $U_\alpha$ by $U_\alpha \cap U$.

As is well known, tangential directions (to the feasible set) play a
fundamental role in optimization. Tangential directions at a point can be
viewed as derivatives of curves passing through that point. When $\MM = \R^n$,
these curves can be taken to be straight curves $t \mapsto \bp + t \, \bv$ of
arbitrary velocity $\bv \in \R^n$. This shows that $\R^n$ serves as its own
tangent space. An adaptation to the setting of a smooth manifold leads to the
following

\begin{definition}[Tangent space]
	\label{definition:tangent_space}
	\begin{enumerate}[label=$(\alph*)$,leftmargin=*,itemsep=\baselineskip]
		\item
			A function $\gamma\colon(-\varepsilon,\varepsilon) \to \MM$ is called a
			$C^1$-curve about~$\bp \in \MM$ if $\gamma(0) = \bp$ holds and
			$\varphi_\alpha \circ \gamma$ is of class $C^1$ for some (equivalently,
			every) chart $(U_\alpha,\varphi_\alpha)$ about~$\bp$.
		\item
			Two $C^1$-curves $\gamma$ and $\zeta$ about~$\bp \in \MM$ are said to be
			equivalent if
			\begin{equation}
				\label{eq:tangent_vector_geometric_definition_equivalence_of_curves}
				\restr{\frac{\d}{\dt} (\varphi_\alpha \circ \gamma)(t)}{t=0}
				=
				\restr{\frac{\d}{\dt} (\varphi_\alpha \circ \zeta)(t)}{t=0}
			\end{equation}
			holds for some (equivalently, every) chart $(U_\alpha,\varphi_\alpha)$
			about~$\bp$.
		\item
			Suppose that $\gamma$ is a $C^1$-curve about~$\bp \in \MM$ and that
			$[\gamma]$ is its equivalence class. Then the following linear map,
			denoted by $[\dot \gamma(0)]$ or $[\frac{\d}{\dt} \gamma(0)]$ and defined
			as
			\begin{equation}
				\label{eq:tangent_vector_generated_by_a_curve}
				[\dot \gamma(0)](f)
				\coloneqq
				\restr{\frac{\d}{\dt} (f \circ \gamma)}{t=0}
			\end{equation}
			takes $C^1$-functions $f\colon U \to \R$ defined in some open
			neighborhood $U \subseteq \MM$ of $\bp$ into $\R$. It is called the tangent
			vector to $\MM$ at $\bp$ along (or generated by) the curve $\gamma$.

		\item
			The collection of all tangent vectors at $\bp$, i.e.,
			\begin{equation}
				\label{eq:tangent_space_on_M}
				\tangentspaceM{\bp}
				\coloneqq
				\bigh\{\}{ [\dot \gamma(0)]\colon[\dot \gamma(0)]
				\text{ is generated by some $C^1$-curve $\gamma$ about $\bp$} }
				,
			\end{equation}
			is termed the tangent space to $\MM$ at $\bp$.
	\end{enumerate}
\end{definition}

\begin{remark}[Tangent space]
	\label{remark:tangent_space}
	\begin{enumerate}[itemsep=\baselineskip,leftmargin=*]
		\item
			We infer from \eqref{eq:tangent_vector_generated_by_a_curve} that the
			tangent vector $[\dot \gamma(0)]$ along the curve $\gamma$ about~$\bp$
			generalizes the notion of the directional derivative operator, acting on
			$C^1$-functions defined near~$\bp$.

		\item
			It can be shown that the tangent space $\tangentspaceM{\bp}$ to $\MM$ at
			$\bp$ is a vector space of dimension $n$ under the operations $\alpha
			\odot [\gamma] = [\alpha \odot \gamma]$ and $[\gamma] \oplus [\zeta] = [\gamma \oplus_\varphi \zeta]$, defined in terms of
			\begin{subequations}
				\label{eq:vector_space_operations_on_curves}
				\begin{align}
					\label{eq:multiplication_by_a_scalar_in_the_geometric_tangent_space}
					&
					\alpha \odot \gamma\colon t \mapsto \gamma(\alpha \, t) \in \MM
					\quad \text{for } \alpha \in \R
					,
					\\
					\label{eq:addition_in_the_geometric_tangent_space}
					&
					\gamma \oplus_{\varphi} \zeta\colon t \mapsto \varphi^{-1} \bigh(){(\varphi \circ \gamma)(t) + (\varphi \circ \zeta)(t) - \varphi(\bp)} \in \MM
				\end{align}
			\end{subequations}
			for arbitrary representatives of their respective equivalence classes. Here
			$\varphi$ is an arbitrary chart about~$\bp$, and its choice does
			not affect the definition of $[\gamma] \oplus [\zeta]$ although it does affect the definition of representative.
	\end{enumerate}
\end{remark}

Finally, we require the generalization of the notion of the derivative for
functions $f\colon\MM \to \R$.
\begin{definition}[Differential]
	\label{definition:differential}
	Suppose that $f\colon\MM \to \R$ is a $C^1$-function and $\bp \in \MM$.
	Then the following linear map, denoted by $(\d f)(\bp)$ and defined as
	\begin{equation}
		\label{eq:differential_of_a_real-valued_function}
		(\d f)(\bp) \, [\dot \gamma(0)]
		\coloneqq
		[\dot \gamma(0)](f)
	\end{equation}
	takes tangent vectors $[\dot \gamma(0)]$ into $\R$.
	It is called the differential of $f$ at $\bp$.
\end{definition}
By definition, the differential $(\d f)(\bp)$ of a real-valued function is a
cotangent vector, i.e., an element from the cotangent space
$\cotangentspaceM{\bp}$, the dual of the tangent space $\tangentspaceM{\bp}$. In
fact, every element of $\cotangentspaceM{\bp}$ is the differential of a
$C^1$-function $s$ at $\bp$. Therefore we denote, without loss of generality,
generic elements of $\cotangentspaceM{\bp}$ by $(\d s)(\bp)$.

\begin{remark}
	\label{remark:notation}
	In the literature on differential geometry the tangent space is usually denoted by $\TT_\bp \MM$ and the cotangent space by $\TT_\bp^* \MM$.
	Moreover the differential of a real-valued function $s$ at $\bp$ is written as $(\d s)_\bp$.
	We hope that our slightly modified notation is more intuitive for readers familiar with nonlinear programming notation.
	We also remark that \cref{definition:differential} easily generalizes to vector valued functions $g: \MM \to \R^m$ by applying \eqref{eq:differential_of_a_real-valued_function} component by component.
\end{remark}

In the following two sections, we are going to derive the KKT theory for
\eqref{eq:nlp} and associated constraint qualifications on smooth manifolds. We
wish to point out that the above notions from differential geometry are
sufficient for these purposes. In particular, we do not need to introduce a
Riemannian metric (a smoothly varying collection of inner products on the
tangent spaces), nor do we need to consider embeddings of $\MM$ into some $\R^N$
for some $N \ge n$. Moreover, we do not need to make further topological
assumptions such as compactness, connectedness, or orientability of $\MM$.

As we mentioned in the introduction, the subsequent results could be derived by transcribing \eqref{eq:nlp} locally into a problem in Euclidean space, using a chart.
This is due to the fact that this transformation leaves the notion of local minimum intact, as shown by the following lemma.
\begin{lemma}[compare \protect{\cite[Sec.~4.1]{YangZhangSong2014}}]
	\label{lemma:local_optimality_in_charts}
	Suppose that $(U,\varphi)$ is an arbitrary chart about~$\bp^*$.
	The following are equivalent:
	\begin{enumerate}[label=$(\alph*)$,leftmargin=*]
		\item
			$\bp^*$ is a local minimizer of \eqref{eq:nlp}.
		\item
			$\varphi(\bp^*)$ is a local minimizer of
			\begin{equation}
				\label{eq:nlp_in_charts}
				\left\{
					\quad
					\begin{aligned}
						\text{Minimize} \quad & (f \circ \varphi^{-1})(x) \text{ w.r.t.\ }  x \in \varphi(U) \subseteq \R^n \\
						\text{s.t.} \quad & (g \circ \varphi^{-1})(x) \le 0 \\
						\text{and} \quad & (h \circ \varphi^{-1})(x) = 0.
					\end{aligned}
				\right.
			\end{equation}
	\end{enumerate}
\end{lemma}
\begin{proof}
	Suppose first that $\bp^* \in \Omega$ is a local minimizer of \eqref{eq:nlp}, i.e., there exists an open neighborhood $U_1$ of $\bp^*$ such that $f(\bp^*) \le f(\bp)$ holds for all $\bp \in U_1 \cap \Omega$.
	We can assume, by shrinking $U_1$ if necessary, that $U_1 \subseteq U$ holds.
	This implies $f(\varphi(\bp^*)) \le f(\varphi(\bp))$ for all $\bp \in U_1 \cap \Omega$.
	Since $\varphi(U_1)$ is an open neighborhood of $\varphi(\bp^*)$, $\varphi(\bp^*)$ is a minimizer of \eqref{eq:nlp_in_charts}.
	The converse is proved similarly.
\end{proof}
However, we are going to prefer working directly with \eqref{eq:nlp} using the language of differential geometry and minimize the explicit use of charts.

\section{KKT Conditions and Constraint Qualifications}
\label{sec:KKT_and_CQs}

In this section we develop first-order necessary optimality conditions in KKT form for \eqref{eq:nlp}.
To begin with, we briefly recall the arguments when $\MM = \R^n$; see for instance \cite[Chap.~12]{NocedalWright2006} or \cite[Chap.~2]{ForstHoffmann2010}.

\subsection{KKT Conditions in $\R^n$}
\label{subsec:KKT_and_CQs_in_Rn}

We define $\Omega \coloneqq \bigh\{\}{ x \in \R^n: g(x) \le 0, \; h(x) = 0 }$ to be the feasible set and associate with \eqref{eq:nlp} the Lagrangian
\begin{equation}
	\label{eq:Lagrangian_Rn}
	\LL(x,\mu,\lambda)
	\coloneqq
	f(x) + \mu \, g(x) + \lambda \, h(x)
	,
\end{equation}
where $\mu \in \R_m$ and $\lambda \in \R_q$.
Using Taylor's theorem, one easily shows that a local minimizer $x^*$ satisfies the necessary optimality condition
\begin{equation}
	\label{eq:non-negative_directional_derivatives_Rn}
	f'(x^*) \, d \ge 0
	\quad
	\text{for all } d \in \tangentconeRn{\Omega}{x^*}
	,
\end{equation}
where $\tangentconeRn{\Omega}{x^*}$ denotes the tangent cone,
\begin{equation}
	\label{eq:tangent_cone_on_Rn}
	\begin{aligned}
		\tangentconeRn{\Omega}{x^*}
		\coloneqq
		\Big\{ d \in \R^n:{} & \text{there exist sequences } (x_k) \subset \Omega, \; x_k \to x^*, \; (t_k) \searrow 0,
			\\
			&
		\text{such that } d = \lim_{k \to \infty} \frac{x_k - x^*}{t_k} \Big\}
		.
	\end{aligned}
\end{equation}
This cone is also known as contingent cone or the Bouligand cone; compare \cite{JimenezNovo2006,Penot1985}.
It is closed but not necessarily convex.
Since $\tangentconeRn{\Omega}{x^*}$ is inconvenient to work with, one introduces the linearizing cone
\begin{equation}
	\label{eq:linearizing_cone_Rn}
	\begin{aligned}
		\linearizingconeRn{\Omega}{x^*}
		\coloneqq
		\big\{ d \in \R^n: \;
			&
			g_i'(x^*) \, d \le 0 \quad \text{for all } i \in \AA(x^*),
			\\
			&
		h_j'(x^*) \, d = 0 \quad \text{for all } j = 1, \ldots, q \big\}
		.
	\end{aligned}
\end{equation}
Here $\AA(x^*) \coloneqq \big\{ 1 \le i \le m: g_i(x^*) = 0 \big\}$ is the index set of active inequalities at $x^*$.
Moreover, $\II(x^*) \coloneqq \{1, \ldots, m\} \setminus \AA(x^*)$ are the inactive inequalities.
It is easy to see that $\linearizingconeRn{\Omega}{x^*}$ is a closed convex cone and that $\tangentconeRn{\Omega}{x^*} \subseteq \linearizingconeRn{\Omega}{x^*}$ holds; see for instance \cite[Lem.~12.2]{NocedalWright2006}.

Using the definition of the polar cone of a set $B \subseteq \R^n$,
\begin{equation}
	\label{eq:polar_cone_of_any_set_Rn}
	B^\circ
	\coloneqq
	\{ s \in \R_n: s \, d \le 0 \text{ for all } d \in B \}
	,
\end{equation}
the first-order necessary optimality condition \eqref{eq:non-negative_directional_derivatives_Rn} can also be written as $-f'(x^*) \in \tangentconeRn{\Omega}{x^*}^\circ$.
Since the polar of the tangent cone is often not easily accessible, one prefers to work with $\linearizingconeRn{\Omega}{x^*}^\circ$ instead, which has the representation
\begin{equation}
	\label{eq:representation_polar_cone_of linearizing_cone_Rn}
	\begin{aligned}
		\linearizingconeRn{\Omega}{x^*}^\circ
		=
		\Big\{ s & = \sum_{i=1}^m \mu_i \, g_i'(x^*) + \sum_{j=1}^q \lambda_j \, h_j'(x^*), \\
		& \mu_i \ge 0 \text{ for } i \in \AA(x^*), \; \mu_i = 0 \text{ for } i \in \II(x^*), \; \lambda_j \in \R \Big\}
		\subseteq \R_n
		,
	\end{aligned}
\end{equation}
as can be shown by means of the Farkas lemma; compare \cite[Lem.~12.4]{NocedalWright2006}.
\ifthenelse{\boolean{ispreprint}}{We state it here in a slightly more general (yet equivalent) form than usual, where $V$ is a finite dimensional vector space and $A \in \LL(V,\R^q)$ is a linear map from $V$ into $\R^q$ for some $q \in \N$.
The adjoint of $A$, denoted by $A^*$, then belongs to $\LL(\R_q,V^*)$, where $V^*$ is the dual space of $V$.

\begin{lemma}[Farkas]
	\label{lemma:Farkas-Lemma}
	Suppose that $V$ is a finite dimensional vector space, $A \in \LL(V,\R^q)$ and $b \in V^*$.
	The following are equivalent:
	\begin{enumerate}[label=$(\alph*)$,leftmargin=*]
		\item
			The system $A^* y = b$ has a solution $y \in \R_q$ which satisfies $y \ge 0$.
		\item
			For any $d \in V$, $A \, d \ge 0$ implies $b \, d \ge 0$.
	\end{enumerate}
\end{lemma}
}{}

Continuing our review, we notice that $\tangentconeRn{\Omega}{x^*} \subseteq \linearizingconeRn{\Omega}{x^*}$ entails $\linearizingconeRn{\Omega}{x^*}^\circ \subseteq \tangentconeRn{\Omega}{x^*}^\circ$, hence \eqref{eq:non-negative_directional_derivatives_Rn} does \emph{not} imply
\begin{equation}
	\label{eq:negative_derivative_belongs_to_polar_of_linearizing_cone_Rn}
	- f'(x^*) \in \linearizingconeRn{\Omega}{x^*}^\circ
	.
\end{equation}
This is where constraint qualifications come into play. The weakest, the Guignard qualification (GCQ), see \cite{Guignard1969:1}, requires the equality $\linearizingconeRn{\Omega}{x^*}^\circ = \tangentconeRn{\Omega}{x^*}^\circ$.
Realizing that \eqref{eq:negative_derivative_belongs_to_polar_of_linearizing_cone_Rn} is nothing but the KKT conditions,
\begin{subequations}
	\label{eq:KKT_conditions_Rn}
	\begin{align}
		\label{eq:KKT_conditions_1_Rn}
		&
		\LL_x(x^*,\mu,\lambda) = f'(x^*) + \mu \, g'(x^*) + \lambda \, h'(x^*) = 0
		,
		\\
		\label{eq:KKT_conditions_2_Rn}
		&
		h(x^*) = 0
		,
		\\
		\label{eq:KKT_conditions_3_Rn}
		&
		\mu \ge 0, \quad g(x^*) \le 0, \quad \mu \, g(x^*) = 0
		,
	\end{align}
\end{subequations}
we obtain the well known
\begin{theorem}
	\label{theorem:first-order_necessary_optimality_conditions_Rn}
	Suppose that $x^*$ is a local minimizer of \eqref{eq:nlp} for $\MM = \R^n$ and that the GCQ holds at $x^*$.
	Then there exist Lagrange multipliers $\mu \in \R_m$, $\lambda \in \R_q$, such that the KKT conditions \eqref{eq:KKT_conditions_Rn} hold.
\end{theorem}

In practice one of course often works with stronger constraint qualifications, which are easier to verify.
We are going to consider in \cref{subsec:CQs_on_M} the analogue of the classical chain $\text{LICQ} \; \Rightarrow \; \text{MFCQ} \; \Rightarrow \; \text{ACQ} \; \Rightarrow \; \text{GCQ}$ on smooth manifolds.

\subsection{KKT Conditions for Optimization Problems on Smooth Manifolds}
\label{subsec:KKT_on_M}

In this section we adapt the argumentation sketched in
\cref{subsec:KKT_and_CQs_in_Rn} to problem \eqref{eq:nlp}, where $\MM$ is a
smooth manifold. Our first result is the analogue of
\cref{theorem:first-order_necessary_optimality_conditions_Rn}, showing that the
GCQ renders the KKT conditions a system of first-order necessary optimality
conditions for local minimizers. For convenience, we summarize in
\cref{tab:summary_of_concepts} how the relevant quantities need to be translated
when moving from $\MM = \R^n$ to manifolds.
\begin{table}
	\centering
	\begin{tabular}{ll}
		\toprule
		$\MM = \R^n$
		&
		$\MM$ smooth manifold
		\\
		\midrule
		tangent space $\R^n$
		&
		tangent space $\tangentspaceM{\bp}$ \eqref{eq:tangent_vector_generated_by_a_curve}
		\\
		tangent cone $\tangentconeRn{\Omega}{x}$ \eqref{eq:tangent_cone_on_Rn}
		&
		tangent cone $\tangentconeM{\Omega}{\bp}$ \eqref{eq:tangent_cone_on_M}
		\\
		linearizing cone $\linearizingconeRn{\Omega}{x}$ \eqref{eq:linearizing_cone_Rn}
		&
		linearizing cone $\linearizingconeM{\Omega}{\bp}$ \eqref{eq:linearizing_cone_M}
		\\
		\midrule
		cotangent space $\R_n$
		&
		cotangent space $\cotangentspaceM{\bp}$
		\\
		derivative $f'(x) \in \R_n$
		&
		differential $(\d f)(\bp) \in \cotangentspaceM{\bp}$ \eqref{eq:differential_of_a_real-valued_function}
		\\
		polar cone $\subseteq \R_n$ \eqref{eq:representation_polar_cone_of linearizing_cone_Rn}
		&
		polar cone $\linearizingconeM{\Omega}{\bp}^\circ \subseteq \cotangentspaceM{\bp}$ \eqref{eq:representation_polar_cone_of linearizing_cone_M}
		\\
		\midrule
		Lagrange multipliers $\mu \in \R_m$, $\lambda \in \R_q$
		&
		same as for $\MM = \R^n$
		\\
		\bottomrule
		\\
	\end{tabular}
	\caption{Summary of concepts related to KKT conditions and constraint qualifications.}
	\label{tab:summary_of_concepts}
\end{table}

Let us denote by
\begin{equation}
	\label{eq:feasible_set_on_M}
	\Omega
	\coloneqq
	\bigh\{\}{ \bp \in \MM: g(\bp) \le 0, \; h(\bp) = 0 }
\end{equation}
the feasible set of \eqref{eq:nlp}.
As in $\R^n$, $\Omega$ is a closed subset of $\MM$ due to the continuity of $g$ and $h$.

A point $\bp^* \in \Omega$ is a local minimizer of \eqref{eq:nlp} if there exists a neighborhood $U$ of $\bp^*$ such that
\begin{equation*}
	f(\bp^*) \le f(\bp)
	\quad \text{for all } \bp \in U \cap \Omega
	.
\end{equation*}

The first notion of interest is the tangent cone at a feasible point.
In view of \eqref{eq:tangent_vector_generated_by_a_curve}, it may be tempting to consider
\begin{equation}
	\label{eq:classical_tangent_cone_on_M}
	\begin{aligned}
		\classicaltangentconeM{\Omega}{\bp}
		\coloneqq
		&
		\bigh\{.{ [\dot \gamma(0)] \in \tangentspaceM{\bp}: [\dot \gamma(0)] \text{ is generated by some $C^1$-curve}}
			\\
			&
			\quad
		\bigh.\}{
		\text{$\gamma$ about~$\bp$ which satisfies } \gamma(t) \in \Omega \text{ for all } t \in [0,\varepsilon) }
		.
	\end{aligned}
\end{equation}
In fact this is the analogue of what is known as the cone of attainable directions and it was used in the original works of~\cite{Karush1939,KuhnTucker1951}.
However, as is well known, this cone is, in general,
strictly smaller than the Bouligand tangent cone \eqref{eq:tangent_cone_on_Rn}
when $\MM = \R^n$; see for instance \cite{Penot1985,JimenezNovo2006},
\cite[Ch.~3.5]{BazaraaShetty1976} and \cite[Ch.~4.1]{AubinFrankowska2009}.

In order to properly generalize the Bouligand tangent cone
\eqref{eq:tangent_cone_on_Rn} to the smooth manifold setting, we
consider sequences rather than curves.
This leads to the following
\begin{definition}[(Bouligand) tangent cone]
	\label{definition:Bouligand_tangent_cone_on_M}
	Suppose that $\bp \in \Omega$ holds.
	\begin{enumerate}[label=$(\alph*)$,leftmargin=*,itemsep=\baselineskip]
		\item
			A tangent vector $[\dot \gamma(0)] \in \tangentspaceM{\bp}$ is called a tangent vector to $\Omega$ at $\bp$ if there exist sequences $(\bp_k) \subseteq \Omega$ and $t_k \searrow 0$ such that for all $C^1$-functions $f$ defined near~$\bp$, we have
			\begin{equation}
				\label{eq:tangent_vector_to_Omega}
				[\dot \gamma(0)](f)
				= 
				\lim_{k \to \infty} \frac{f(\bp_k) - f(\bp)}{t_k}
				.
			\end{equation}
			We refer to the sequence $(\bp_k,t_k)$ as a tangential sequence to $\Omega$ at $\bp$.

		\item
			The collection of all tangent vectors to $\Omega$ at $\bp$ is termed the (Bouligand) tangent cone to $\Omega$ at $\bp$ and denoted by 
			\begin{equation}
				\label{eq:tangent_cone_on_M}
				\tangentconeM{\Omega}{\bp}
				\coloneqq
				\{[\dot \gamma(0)] \in \tangentspaceM{\bp} : [\dot \gamma(0)] \text{ is a tangent vector to $\Omega$ at $\bp$} \}
				.
			\end{equation}
	\end{enumerate}
\end{definition}

The following proposition shows that \eqref{eq:tangent_cone_on_M} could also have been defined as a lifting via the chart differential of the classical tangent cone to the chart image of the feasible set near $\bp$.
	This was in fact used as the definition of the tangent cone in \cite[eq.~(3.7)]{YangZhangSong2014}.

\begin{proposition}
	\label{lemma:comparison_of_tangent_cone_on_M_with_tangent_cone_through_chart}
		Suppose that $\bp \in \Omega$, and let $(U,\varphi)$ be a chart about $\bp$.
		Then 
	
		\begin{equation}
			\label{eq:comparison_of_tangent_cone_on_M_with_tangent_cone_through_chart}
			\bigh(){(\d \varphi)(\bp)} \, \tangentconeM{\Omega}{\bp}
			=
			\tangentconeRn{\varphi(U \cap \Omega)}{\varphi(\bp)}
			.
		\end{equation}
\end{proposition}
\begin{proof}
		We divide the proof into two parts and first prove \eqq{$\supset$} in \eqref{eq:tangent_vector_to_Omega}.
		To this end, suppose that $d \in \tangentconeRn{\varphi(U \cap \Omega)}{\varphi(\bp)}$, i.e., there exist sequences $(x_k) \subset \varphi(U \cap \Omega)$, $x_k \to \varphi(\bp) \eqqcolon x$ and $t_k \searrow 0$, such that $d = \lim_{k \to \infty} (x_k - x)/t_k$; see \eqref{eq:tangent_cone_on_Rn}.
		Define $\bp_k \coloneqq \varphi^{-1}(x_k) \in U \cap \Omega$ and $\bp \coloneqq \varphi^{-1}(x) \in U \cap \Omega$.
		Then $\bp_k \to \bp$ since $\varphi^{-1}$ is continuous.
		Further, define a curve $\gamma$ via $\gamma(t) \coloneqq \varphi^{-1}(\varphi(\bp) + t \, d)$ for $\abs{t}$ sufficiently small.
		We show that $[\dot \gamma(0)]$ belongs to $\tangentconeM{\Omega}{\bp}$ by verifying \eqref{eq:tangent_vector_to_Omega}.
		To this end, let $f$ be an arbitrary $C^1$-function defined near $\bp$.
		Then we have
	\begin{equation*}
		[\dot \gamma(0)](f)
		=
		\restr{\frac{\d}{\dt} (f \circ \gamma)}{t=0}
		=
		\restr{\frac{\d}{\dt} \bigh(){(f \circ \varphi^{-1})(\varphi(\bp) + t \, d)}}{t=0}
		=
		(f \circ \varphi^{-1})'(\varphi(\bp)) \, d
	\end{equation*}
		by the definition of $\gamma$ and the chain rule.
		On the other hand, 
	\begin{equation*}
		\lim_{k \to \infty} \frac{f(\bp_k) - f(\bp)}{t_k}
		=
		\lim_{k \to \infty} \frac{(f \circ \varphi^{-1})(x_k) - (f \circ \varphi^{-1})(x)}{t_k}
		=
		(f \circ \varphi^{-1})'(\varphi(\bp)) \, d
	\end{equation*}
	holds, which proves \eqref{eq:tangent_vector_to_Omega} and thus $[\dot \gamma(0)] \in \tangentconeM{\Omega}{\bp}$.
		By \cref{definition:differential}, \cref{remark:notation}, \eqref{eq:tangent_vector_generated_by_a_curve} and the definition of $\gamma$, we have
	
	\begin{equation*}
		(\d \varphi)(\bp) \, [\dot \gamma(0)]
		=
		[\dot \gamma(0)](\varphi)
		=
		\restr{\frac{\d}{\dt} (\varphi \circ \gamma)}{t=0}
		=
		\restr{\frac{\d}{\dt} \bigh(){(\varphi \circ \varphi^{-1})(\varphi(\bp) + t \, d)}}{t=0}
		=
		d
		.
	\end{equation*}
	This confirms $d \in \bigh(){(\d \varphi)(\bp)} \, \tangentconeM{\Omega}{\bp}$ and thus the first part of the proof.

	For the reverse inequality \eqq{$\subset$}, we begin with an element $[\dot \gamma(0)] \in \tangentconeM{\Omega}{\bp}$ and an associated tangential sequence $(\bp_k,t_k)$ as in \eqref{eq:tangent_vector_to_Omega}.
		Again by \cref{definition:differential,remark:notation}, we obtain
	
	\begin{equation*}
		\bigh(){(\d \varphi)(\bp)} \, \tangentconeM{\Omega}{\bp}
		=
		[\dot \gamma(0)](\varphi)
		=
		\lim_{k \to \infty} \frac{\varphi(\bp_k) - \varphi(\bp)}{t_k}
	\end{equation*}
	and the limit exists by \eqref{eq:tangent_vector_to_Omega}.
	The sequence $\varphi(\bp_k, t_k)$ satisfies all the requirements to generate an element of $\tangentconeRn{\varphi(U \cap \Omega)}{\varphi(\bp)}$, compare \eqref{eq:tangent_cone_on_Rn}.
\end{proof}

\begin{remark}[Tangent cone]
	\label{remark:Bouligand_tangent_cone_on_M}
			The notion of tangent vectors to subsets of smooth manifolds can be traced back to \cite[Def.~2.1]{MotreanuPavel1982}, where they were called \emph{quasi-tangent} vectors and introduced, in our notation, as vectors $[\dot \gamma(0)] \in \tangentspaceM{\bp}$ satisfying
			\begin{equation*}
				\lim_{h \to 0} \frac{1}{h} \dist\bigh(){\varphi(\bp) + h \, (D\varphi)(\bp) \, [\dot \gamma(0)], \; \varphi(U \cap \Omega)} = 0
				.
			\end{equation*}
			Here $(U,\varphi)$ is a chart about $\bp$, $(D \varphi)(\bp)$ is the derivative (push-forward) of $\varphi$ at $\bp$, and $\dist$ denotes the (Euclidean) distance between a point and a set in $\R^n$.
			It is straightforward to show that this definition is equivalent to \eqref{eq:tangent_cone_on_M}.
			\ifthenelse{\boolean{ispreprint}}{However we explicitly utilize tangential sequences in the following, and particularly in \cref{lemma:properties_of_the_tangent_cone_on_M}, \cref{theorem:first-order_necessary_optimality_condition_on_M}, and \cref{lemma:relation_between_cones_M}.}{}
\end{remark}

\begin{lemma}[Properties of the tangent cone]
	\label{lemma:properties_of_the_tangent_cone_on_M}
	For any $\bp \in \Omega$, the tangent cone $\tangentconeM{\Omega}{\bp}$ is a closed cone in the tangent space $\tangentspaceM{\bp}$.
\end{lemma}
\begin{proof}
	The result follows immediately from \cref{lemma:comparison_of_tangent_cone_on_M_with_tangent_cone_through_chart} since $\tangentconeRn{\varphi(U \cap \Omega)}{\varphi(\bp)}$is a closed cone in $\R^n$ and $(\d \varphi)(\bp)$ is a bijective, linear map between the vector spaces $\tangentspaceM{\bp}$ and $\R^n$.
	\ifthenelse{\boolean{ispreprint}}{%
			However, we also give a direct proof here.
			Suppose that $[\dot \gamma(0)]$ is an element of the tangent cone $\tangentconeM{\Omega}{\bp}$, associated with the tangential sequence $(\bp_k,t_k)$ as in \eqref{eq:tangent_vector_to_Omega}.
				Let $\alpha > 0$.
			It is easy to see that the curve $\alpha \odot \gamma$ generates $\alpha \, [\dot \gamma(0)]$ and that it is associated with the tangential sequence $(\bp_k,\alpha \, t_k)$.
			This shows that $\tangentconeM{\Omega}{\bp}$ is a cone.

			Let us now confirm that $\tangentconeM{\Omega}{\bp}$ is closed in $\tangentspaceM{\bp}$.
			To this end, consider a sequence $[\dot \gamma_\ell(0)]$ of tangent vectors to $\Omega$ at $\bp$ which converges to a tangent vector $[\dot \gamma(0)] \in \tangentspaceM{\bp}$.
			Each $[\dot \gamma_\ell(0)]$ is associated with a tangential sequence $(\bp_{k,\ell},t_{k,\ell})$, $k \in \N$.
			We need to show that the limit $[\dot \gamma(0)]$ is also associated with a tangential sequence.
			To this end, fix an arbitrary chart $\varphi$ about $\bp$.
			Then by definition, there exist vectors $d_\ell \in \R^n$ such that $(\varphi(\bp_{k,\ell})-\varphi(\bp))/t_{k,\ell} \to d_\ell = [\dot \gamma_\ell(0)](\varphi)$ holds as $k \to \infty$.
			By assumption, $d_\ell \to d \coloneqq [\dot \gamma(0)](\varphi)$ holds.
			Let us now construct a tangential sequence associated with $[\dot \gamma(0)]$.
			For every $\ell \in \N$, we can select an index $k(\ell)$ such that
			\begin{equation*}
				\bigabs{\varphi(\bp_{k(\ell),\ell}) - \varphi(\bp)}_2 \le \frac{1}{\ell},
				\quad
				0 < t_{k(\ell),\ell} \le \frac{1}{\ell},
				\quad
				\Bigabs{ \frac{\varphi(\bp_{k(\ell),\ell}) - \varphi(\bp)}{t_{k(\ell),\ell}}  - d_\ell}_2 \le \frac{1}{\ell}
			\end{equation*}
			holds.
			Consider now the \eqq{diagonal} sequence $(\widehat \bp_\ell, \widehat t_\ell) \coloneqq (\bp_{k(\ell),\ell},t_{k(\ell),\ell})$.
			Obviously $\widehat \bp_\ell$ belongs to $\Omega$, $\widehat t_\ell \searrow 0$ holds and 
			\begin{equation*}
				\Bigabs{\frac{\varphi(\widehat \bp_\ell)-\varphi(\bp)}{\widehat t_\ell} - d}_2
				\le
				\Bigabs{\frac{\varphi(\widehat \bp_\ell)-\varphi(\bp)}{\widehat t_\ell} - d_\ell}_2
				+
				\bigabs{d_\ell - d}_2
				\to 0 \text{ as } \ell \to \infty.
			\end{equation*}
			This shows that 
			\begin{equation*}
				d = [\dot \gamma(0)](\varphi) = \lim_{\ell \to \infty} \frac{\varphi(\widehat \bp_\ell)-\varphi(\bp)}{\widehat t_\ell},
			\end{equation*}
			which is \eqref{eq:tangent_vector_to_Omega} with $f = \varphi$. 
				It remains to confirm that \eqref{eq:tangent_vector_to_Omega} actually holds for all $C^1$-function $f$ defined near~$\bp$.
			However this follows easily by the chain rule as in the proof of \cref{lemma:comparison_of_tangent_cone_on_M_with_tangent_cone_through_chart}.
		}{}
	\end{proof}

The analogue of \eqref{eq:non-negative_directional_derivatives_Rn} is the following
\begin{theorem}[First-order necessary optimality condition]
	\label{theorem:first-order_necessary_optimality_condition_on_M}
	Suppose that $\bp^* \in \Omega$ is a local minimizer of \eqref{eq:nlp}.
	Then we have
	\begin{equation}
		\label{eq:non-negative_directional_derivatives_M}
		[\dot \gamma(0)](f) \ge 0
	\end{equation}
	for all tangent vectors $[\dot \gamma(0)] \in \tangentconeM{\Omega}{\bp^*}$.
\end{theorem}
\begin{proof}
	Suppose that $[\dot \gamma(0)] \in \tangentconeM{\Omega}{\bp^*}$ and that $(\bp_k,t_k)$ is an associated tangential sequence.
	Then we have by local optimality of $\bp^*$
	\begin{equation*}
		\begin{aligned}
			&
			0
			\le
			\frac{f(\bp_k)-f(\bp^*)}{t_k}
			&
			&
			\text{for sufficiently large $k \in \N$}
			\\
			\Rightarrow
			\quad
			&
			0
			\le
			[\dot \gamma(0)](f)
			&
			&
			\text{by \eqref{eq:tangent_vector_to_Omega}}
			.
		\end{aligned}
	\end{equation*}
	This concludes the proof.
\end{proof}

Next we introduce the concept of the linearizing cone \eqref{eq:linearizing_cone_Rn} in the tangent space, similar to \cite[Def.~4.1]{YangZhangSong2014}.
\begin{definition}[Linearizing cone]
	\label{definition:linearizing_cone_M}
	For any $\bp \in \Omega$, we define the linearizing cone to the feasible set $\Omega$ by
	\begin{equation}
		\label{eq:linearizing_cone_M}
		\begin{aligned}
			\linearizingconeM{\Omega}{\bp}
			\coloneqq
			\big\{ [\dot \gamma(0)] \in \tangentspaceM{\bp}: {}
				&
				[\dot \gamma(0)](g^i) \le 0 \quad \text{for all } i \in \AA(\bp),
				\\
				&
				[\dot \gamma(0)](h^j) = 0 \quad \text{for all } j = 1, \ldots, q
			\big\}
			.
		\end{aligned}
	\end{equation}
\end{definition}
As in \cref{subsec:KKT_and_CQs_in_Rn}, $\AA(\bp) \coloneqq \big\{ 1 \le i \le m:
g^i(\bp) = 0 \big\}$ is the index set of active inequalities at $\bp$, and
$\II(\bp) \coloneqq \{1, \ldots, m\} \setminus \AA(\bp)$ are the inactive
inequalities.
Notice that, as is customary in differential geometry, we denote the components
of the vector-valued functions $g$ and $h$ by upper indices.

\begin{remark}
	\label{remark:linearizing_cone_M}
	The linearizing cone could be defined alternatively as
	\begin{equation}
		\label{eq:comparison_of_linearizing_cone_on_M_with_linearizing_cone_through_chart}
			\bigh(){(\d \varphi)(\bp)} \, \linearizingconeM{\Omega}{\bp}
			=
			\linearizingconeRn{\varphi(U \cap \Omega)}{\varphi(\bp)},
	\end{equation}
	compare \cref{lemma:comparison_of_tangent_cone_on_M_with_tangent_cone_through_chart} for the parallel result for the tangent cone.
\end{remark}

\begin{lemma}[Relation between the cones]
	\label{lemma:relation_between_cones_M}
	For any $\bp \in \Omega$, $\linearizingconeM{\Omega}{\bp}$ is a convex cone, and $\tangentconeM{\Omega}{\bp} \subseteq \linearizingconeM{\Omega}{\bp}$ holds.
\end{lemma}
\begin{proof}
	The result follows immediately from \eqref{eq:comparison_of_linearizing_cone_on_M_with_linearizing_cone_through_chart} and the corresponding result in $\R^n$.
	\ifthenelse{\boolean{ispreprint}}{However, we also give a direct proof here.
	To show that $\linearizingconeM{\Omega}{\bp}$ is a convex cone, let $\gamma_1$ and $\gamma_2$ be two curves about~$\bp$, generating the elements $[\dot \gamma_1(0)]$ and $[\dot \gamma_2(0)]$ in $\linearizingconeM{\Omega}{\bp}$, and let $\alpha_1, \alpha_2 \geq 0$.
	Since $\tangentspaceM{\bp}$ is a vector space under $\odot$ and $\oplus$, we have
	\begin{equation*}
		\begin{aligned}
			&
			[(\alpha_1 \odot \gamma_1) \oplus_{\varphi} (\alpha_2 \odot \gamma_2)](g^i)
			=
			\alpha_1 \, [\dot \gamma_1(0)](g^i) + \alpha_2 \, [\dot \gamma_2(0)](g^i)
			\le
			0
			&
			&
			\text{for } i \in \AA(\bp)
			,
			\\
			&
			[(\alpha_1 \odot \gamma_1) \oplus_{\varphi} (\alpha_2 \odot \gamma_2)](h^j)
			=
			\alpha_1 \, [\dot \gamma_1(0)](h^j) + \alpha_2 \, [\dot \gamma_2(0)](h^j)
			=
			0
			&
			&
			\text{for } j = 1, \ldots, q
			,
		\end{aligned}
	\end{equation*}
	hence $[(\alpha_1 \odot \gamma_1) \oplus_{\varphi} (\alpha_2 \odot \gamma_2)]$ belongs to $\linearizingconeM{\Omega}{\bp}$ as well.

	Now let $[\dot \gamma(0)] \in \tangentconeM{\Omega}{\bp}$ be associated with the tangential sequence $(\bp_k,t_k)$ to $\Omega$ at $\bp$.
	Recall that the points $\bp_k$ are feasible.
	Consequently, for $i \in \AA(\bp)$ and $k \in \N$ we have
	\begin{equation*}
		0
		\ge
		\frac{g^i(\bp_k) - g^i(\bp)}{t_k}
		\quad
		\Rightarrow
		\quad
		[\dot \gamma(0)](g^i)
		\le
		0
		.
	\end{equation*}
	Similarly, we get $[\dot \gamma(0)](h^j) = 0$ for $j = 1, \ldots, q$.
	This shows $[\dot \gamma(0)] \in \linearizingconeM{\Omega}{\bp}$.
}{}
\end{proof}

Similar to \eqref{eq:polar_cone_of_any_set_Rn}, the polar cone to a subset $B \subseteq \tangentspaceM{\bp}$ of the tangent space is defined as
\begin{equation}
	\label{eq:polar_cone_of_any_set_M}
	B^\circ
	\coloneqq
	\bigh\{\}{ (\d s)(\bp) \in \cotangentspaceM{\bp}: (\d s)(\bp) \, [\dot \gamma(0)] \le 0 \text{ for all } [\dot \gamma(0)] \in B }
	.
\end{equation}
Let us calculate a representation of $\linearizingconeM{\Omega}{\bp}^\circ$, similar to \eqref{eq:representation_polar_cone_of linearizing_cone_Rn}.

\begin{lemma}
	\label{lemma:representation_polar_cone_of linearizing_cone_M}
	For any $\bp \in \Omega$, we have
	\begin{align}
		\linearizingconeM{\Omega}{\bp}^\circ
		=
		\Big\{ (\d s)(\bp) & = \sum_{i=1}^m \mu_i \, (\d g^i)(\bp) + \sum_{j=1}^q \lambda_j \, (\d h^j)(\bp),
			\notag
			\\
		& \mu_i \ge 0 \text{ for } i \in \AA(\bp), \; \mu_i = 0 \text{ for } i \in \II(\bp), \; \lambda_j \in \R \Big\}
		\subseteq \cotangentspaceM{\bp}
		.
		\label{eq:representation_polar_cone_of linearizing_cone_M}
	\end{align}
\end{lemma}
\begin{proof}
	It is easy to see that for vector spaces $V$ and $W$ of the finite dimension and bijective, linear $A: V \to W$, we have $(A^{-1} K)^\circ = A^* K^\circ$ in $V^*$ for all $K \subset W$.
		Here $V^*$ and $W^*$ are the dual spaces of $V$ and $W$ and $A^*: W^* \to V^*$ is the adjoint map.
	We apply this with $K = \linearizingconeRn{\varphi(U \cap \Omega)}{\varphi(\bp)} \subset W = \R^n$, $V = \tangentspaceM{\bp}$ and $A = (\d \varphi)(\bp)$ to obtain
	\begin{equation*}
		\begin{aligned}
			\MoveEqLeft
			\linearizingconeM{\Omega}{\bp}^\circ
			\\
			& 
			=
			\Bigh(){\bigh(){(\d \varphi)(\bp)}^{-1} \linearizingconeRn{\varphi(U \cap \Omega)}{\varphi(\bp)}}^\circ
			& & \text{by \eqref{eq:comparison_of_linearizing_cone_on_M_with_linearizing_cone_through_chart}}
			\\
			&
			=
			\bigh(){(\d \varphi)(\bp)}^* \linearizingconeRn{\varphi(U \cap \Omega)}{\varphi(\bp)}^\circ
			\\
			&
			=
			\bigh(){(\d \varphi)(\bp)}^* 
			\Big\{ \sum_{i=1}^m \mu_i \, (g^i \circ \varphi^{-1})'(\varphi(\bp)) + \sum_{j=1}^q \lambda_j \, (h^j \circ \varphi^{-1})'(\varphi(\bp)), 
				\\
				& 
				\qquad \qquad 
			\mu_i \ge 0 \text{ for } i \in \AA(\varphi(\bp)), \; \mu_i = 0 \text{ for } i \in \II(\varphi(\bp)), \; \lambda_j \in \R \Big\}
			& & \text{by \eqref{eq:representation_polar_cone_of linearizing_cone_Rn}}
			\\
			&
			=
			\Big\{ \sum_{i=1}^m \mu_i \, (\d g^i)(\bp) + \sum_{j=1}^q \lambda_j \, (\d h^j)(\bp),
				\\
				& 
				\qquad \qquad 
			\mu_i \ge 0 \text{ for } i \in \AA(\bp), \; \mu_i = 0 \text{ for } i \in \II(\bp), \; \lambda_j \in \R \Big\}
			.
		\end{aligned}
	\end{equation*}
	The last equality follows from the chain rule applied to $(g^i \circ \varphi^{-1}) \circ \varphi$.
	\ifthenelse{\boolean{ispreprint}}{We also give an alternative, direct proof here using the Farkas \cref{lemma:Farkas-Lemma}.
	When $(\d s)(\bp)$ belongs to the set on the right-hand side of \cref{eq:representation_polar_cone_of linearizing_cone_M} and $[\dot \gamma(0)] \in \linearizingconeM{\Omega}{\bp}$ is arbitrary, then
	\begin{equation*}
		\begin{aligned}
			(\d s)(\bp) \, [\dot \gamma(0)]
			&
			=
			\sum_{i=1}^m \mu_i \, (\d g^i)(\bp) [\dot \gamma(0)] + \sum_{j=1}^q \lambda_j \, (\d h^j)(\bp) [\dot \gamma(0)]
			\\
			&
			=
			\sum_{i=1}^m \mu_i \, [\dot \gamma(0)](g^i) + \sum_{j=1}^q \lambda_j \, [\dot \gamma(0)](h^j)
		\end{aligned}
	\end{equation*}
	by definition of the differential; see \eqref{eq:differential_of_a_real-valued_function}.
	Utilizing the sign conditions in \eqref{eq:representation_polar_cone_of linearizing_cone_M} and the definition of $\linearizingconeM{\Omega}{\bp}$ in \eqref{eq:linearizing_cone_M} shows $(\d s)(\bp) \, [\dot \gamma(0)] \le 0$, i.e., $(\d s)(\bp) \in \linearizingconeM{\Omega}{\bp}^\circ$.

	For the converse, consider the linear map
	\begin{equation*}
		A
		\coloneqq
		\left(
			\begin{array}{l@{}l@{}l}
				- & (\d g^i)(\bp) & \big|_{i \in \AA(\bp)} \\
				- & (\d h^j)(\bp) & \big|_{j = 1, \ldots, q} \\
				& (\d h^j)(\bp) & \big|_{j = 1, \ldots, q}
			\end{array}
		\right)
	\end{equation*}
	which maps the tangent space $\tangentspaceM{\bp}$ into $\R^{r}$, where $r = \abs{\AA(\bp)} + 2 \, q$.
	By \eqref{eq:linearizing_cone_M}, $[\dot \gamma(0)] \in \linearizingconeM{\Omega}{\bp}$ holds if and only if $A \, [\dot \gamma(0)] \ge 0$.

	Now let $(\d s)(\bp) \in \linearizingconeM{\Omega}{\bp}^\circ$, i.e., $(\d s)(\bp) \, [\dot \gamma(0)] \le 0$ holds for all $[\dot \gamma(0)]$ such that $A \, [\dot \gamma(0)] \ge 0$.
	The Farkas \cref{lemma:Farkas-Lemma} (with $V = \tangentspaceM{\bp}$ and $b = -(\d s)(\bp)$) shows that $A^* y = -(\d s)(\bp)$ has a solution $y \in \R_q$, $y \ge 0$.
	Now split $y \eqqcolon (\mu_{|\AA(\bp)},\lambda^+,\lambda^-)$, set $\lambda \coloneqq \lambda^+ - \lambda^-$ and pad $\mu$ by setting $\mu_{|\II(\bp)} \coloneqq 0$.
	This shows that $(\d s)(\bp)$ indeed has the representation postulated in \eqref{eq:representation_polar_cone_of linearizing_cone_M}.
}{}
\end{proof}

We associate with \eqref{eq:nlp} the Lagrangian
\begin{equation}
	\label{eq:Lagrangian_M}
	\LL(\bp,\mu,\lambda)
	\coloneqq
	f(\bp) + \mu \, g(\bp) + \lambda \, h(\bp)
	,
\end{equation}
where $\mu \in \R_m$ and $\lambda \in \R_q$, and the KKT conditions
\begin{subequations}
	\label{eq:KKT_conditions_M}
	\begin{align}
		\label{eq:KKT_conditions_1_M}
		&
		(\d \LL)(\bp,\mu,\lambda) = (\d f)(\bp) + \mu \, (\d g)(\bp) + \lambda \, (\d h)(\bp) = 0
		,
		\\
		\label{eq:KKT_conditions_2_M}
		&
		h(\bp) = 0
		,
		\\
		\label{eq:KKT_conditions_3_M}
		&
		\mu \ge 0, \quad g(\bp) \le 0, \quad \mu \, g(\bp) = 0
		.
	\end{align}
\end{subequations}
Here we introduced for convenience of notation the differential of the vector-valued functions $g = (g^1, \ldots, g^m)\transp$
\begin{equation*}
	(\d g)(\bp)
	\coloneqq
	\begin{pmatrix}
		(\d g^1)(\bp) \\ \vdots \\ (\d g^m)(\bp)
	\end{pmatrix}
\end{equation*}
and similarly for $h$.

Just as in the case of $\MM = \R^n$, it is easy to see by \cref{lemma:representation_polar_cone_of linearizing_cone_M} that the KKT conditions \eqref{eq:KKT_conditions_M} are equivalent to
\begin{equation}
	\label{eq:negative_derivative_belongs_to_polar_of_linearizing_cone_M}
	- (\d f)(\bp) \in \linearizingconeM{\Omega}{\bp}^\circ
	.
\end{equation}
We thus obtain the analogue of \cref{theorem:first-order_necessary_optimality_conditions_Rn}:
\begin{theorem}
	\label{theorem:first-order_necessary_optimality_conditions_M}
	Suppose that $\bp^*$ is a local minimizer of \eqref{eq:nlp} and that the GCQ $\linearizingconeM{\Omega}{\bp^*}^\circ = \tangentconeM{\Omega}{\bp^*}^\circ$ holds at $\bp^*$.
	Then there exist Lagrange multipliers $\mu \in \R_m$, $\lambda \in \R_q$, such that the KKT conditions \eqref{eq:KKT_conditions_M} hold.
\end{theorem}

\subsection{Constraint Qualifications for Optimization Problems on Smooth Manifolds}
\label{subsec:CQs_on_M}

In this section we introduce the constraint qualifications (CQ) of linear independence (LICQ), Mangasarian--Fromovitz (MFCQ), Abadie (ACQ) and Guignard (GCQ) and show that the chain of implications
\begin{equation}
	\label{eq:chain_of_implications}
	\text{LICQ} \; \Rightarrow \; \text{MFCQ} \; \Rightarrow \; \text{ACQ} \; \Rightarrow \; \text{GCQ}
\end{equation}
continues to hold in the smooth manifold setting. 
Except for LICQ, which has been used in \cite{YangZhangSong2014}, this is the first time these conditions are being formulated and utilized on smooth manifolds.

\begin{definition}[Constraint qualifications]
	\label{definition:CQs}
	Suppose that $\bp \in \Omega$ holds.
	We define the following constraint qualifications at $\bp$.
	\begin{enumerate}[label=$(\alph*)$,leftmargin=*,itemsep=\baselineskip]
		\item
			\label{item:LICQ}
			The LICQ holds at $\bp$ if $\{ (\d h^j)(\bp) \}_{j=1}^q \cup \{ (\d g^i)(\bp)\}_{i \in \AA(\bp)}$ is a linearly independent set in the cotangent space $\cotangentspaceM{\bp}$.

		\item
			\label{item:MFCQ}
			The MFCQ holds at $\bp$ if $\{ (\d h^j)(\bp) \}_{j=1}^q$ is a linearly independent set and if there exists a tangent vector $[\dot \gamma(0)]$ (termed an MFCQ vector) such that
			\begin{equation}
				\label{eq:MFCQ_conditions_M}
				\begin{aligned}
					(\d g^i)(\bp) [\dot \gamma(0)]
					&
					<
					0
					&
					&
					\text{for all } i \in \AA(\bp)
					,
					\\
					(\d h^j)(\bp) [\dot \gamma(0)]
					&
					=
					0
					&
					&
					\text{for all } j = 1, \ldots, q
					.
				\end{aligned}
			\end{equation}

		\item
			\label{item:ACQ}
			The ACQ holds at $\bp$ if $\linearizingconeM{\Omega}{\bp} = \tangentconeM{\Omega}{\bp}$.

		\item
			\label{item:GCQ}
			The GCQ holds at $\bp$ if $\linearizingconeM{\Omega}{\bp}^\circ = \tangentconeM{\Omega}{\bp}^\circ$.
	\end{enumerate}
\end{definition}

\begin{remark}
	\label{remark:CQs}
	The constraint qualifications in \cref{definition:CQs} are equivalent to their respective counterparts for the local transcription of \eqref{eq:nlp} into an optimization problem in Euclidean space, see \eqref{eq:nlp_in_charts}.
	For instance, when $\varphi$ is a chart about $\bp \in \Omega$, then the LICQ is equivalent to the linear independence of the derivatives $\{ (h^j \circ \varphi^{-1})'(\varphi(\bp)) \}_{j=1}^q \cup \{ (g^i \circ \varphi^{-1})'(\varphi(\bp))\}_{i \in \AA(\varphi(\bp))}$.
	A similar statement holds for the MFCQ, ACQ, and GCQ.
	The result \eqref{eq:chain_of_implications} can therefore be shown by invoking the corresponding statement for \eqref{eq:nlp_in_charts}.
	However, we provide also direct proofs in \cref{proposition:LICQ_implies_MFCQ_M,proposition:MFCQ_implies_ACQ_M}.
\end{remark}

\begin{proposition}
	\label{proposition:LICQ_implies_MFCQ_M}
	LICQ implies MFCQ.
\end{proposition}
\begin{proof}
	Consider the linear system
	\begin{equation*}
		A \, [\dot \gamma(0)]
		\coloneqq
		\left(
			\begin{array}{l@{}l@{}l}
				& (\d g^i)(\bp) & \big|_{i \in \AA(\bp)} \\
				& (\d h^j)(\bp) & \big|_{j = 1, \ldots, q}
			\end{array}
		\right)
		[\dot \gamma(0)]
		=
		(-1, \ldots, -1, 0, \ldots, 0)\transp
		.
	\end{equation*}
	Since the linear map $A$ is surjective by assumption, this system is solvable, and $[\dot \gamma(0)]$ satisfies the MFCQ conditions.
\end{proof}

In order to show that MFCQ implies ACQ, we first prove the following result; compare \cite[Lem.~2.37]{GeigerKanzow2002}.
\begin{proposition}
	\label{proposition:Ljusternik}
	Suppose that $\bp \in \Omega$ and that the MFCQ holds at $\bp$ with the MFCQ
	vector $[\dot \gamma(0)]$. Then the curve $\gamma$ about~$\bp$ which
	generates $[\dot \gamma(0)]$ can be chosen to satisfy the following:
	\begin{enumerate}[label=$(\alph*)$,leftmargin=*]
		\item
			$h^j(\gamma(t)) = 0$ for all $t \in (-\varepsilon,\varepsilon)$ and all $j
			= 1, \ldots, q$.

		\item
			$\gamma(t) \in \Omega$ for all $t \in [0,\varepsilon)$ and even
			$g^i(\gamma(t)) < 0$ for all $t \in (0,\varepsilon)$ and all $i = 1, \ldots,
			m$.
	\end{enumerate}
\end{proposition}
\begin{proof}
	Choose a chart $\varphi$ about~$\bp$ and set $x_0 \coloneqq \varphi(\bp)$.
	We start with an arbitrary $C^1$-curve $\zeta$ about~$\bp$ which generates the MFCQ vector $[\dot \gamma(0)]$.
	We are going to define, in the course of the proof, an alternative $C^1$-curve $\gamma$ about~$\bp$ which generates the same tangent vector and which satisfies the conditions stipulated.

	In the absence of equality constraints ($q = 0$), we can simply take $\gamma = \zeta$.
	Suppose now that $q \ge 1$ holds.
	For some $\varepsilon > 0$, $\zeta(t)$ belongs to the domain of $\varphi$ whenever $t \in (-\varepsilon,\varepsilon)$.
	Define
	\begin{equation*}
		H(y,t)
		\coloneqq
		(h \circ \varphi^{-1}) \varh(){ (\varphi \circ \zeta)(t) + (h \circ \varphi^{-1})'(x_0)\transp y }
		,
		\quad
		(y,t) \in \R^q \times (-\varepsilon,\varepsilon)
		.
	\end{equation*}
	Then $H(0,0) = (h \circ \varphi^{-1})(x_0 + 0) = h(\bp) = 0$ holds.
	Moreover, by the chain rule, the Jacobian of $H$ w.r.t.\ $y$ is
	\begin{equation*}
		H_y(y,t)
		=
		(h \circ \varphi^{-1})' \varh(){ (\varphi \circ \zeta)(t) + (h \circ \varphi^{-1})'(x_0)\transp y } \, (h \circ \varphi^{-1})'(x_0)\transp,
	\end{equation*}
	and in particular, $H_y(0,0) = (h \circ \varphi^{-1})'(x_0) \, (h \circ \varphi^{-1})'(x_0)\transp$. Since $\{ (\d h^j)(\bp) \}_{j=1}^q$ is a linearly independent set of cotangent vectors, the $q \times n$-matrix $(h \circ \varphi^{-1})'(x_0)$ has rank $q$.
	To see this, consider the tangent vectors along the curves $t \mapsto \gamma_k(t) \coloneqq \varphi^{-1}(\varphi(\bp) + t \, e_k)$ for $k = 1, \ldots, n$.
	The entry $(j,k)$ of $(h \circ \varphi^{-1})'(x_0)$ equals $(\d h^j)(\bp) \, [\dot \gamma_k(0)] = \restr{\frac{\d}{\dt} (h^j \circ \gamma_k)(t)}{t=0}$.
	Since the tangent vectors $\{[\dot \gamma_k(0)]\}_{k=1}^n$ are linearly independent and the cotangent vectors $\{ (\d h^j)(\bp) \}_{j=1}^q$ as well, the matrix $(h \circ \varphi^{-1})'(x_0)$ has full rank as claimed.
	This shows that $H_y(0,0)$ is symmetric positive definite.
	Moreover,
	\begin{equation*}
		H_t(y,t)
		=
		(h \circ \varphi^{-1})' \varh(){ (\varphi \circ \zeta)(t) + (h \circ \varphi^{-1})'(x_0)\transp y } \, (\varphi \circ \zeta)'(t)
		,
	\end{equation*}
	whence $H_t(0,0) = (h \circ \varphi^{-1})'(x_0) \, (\varphi \circ \zeta)'(0) = (h \circ \zeta)'(0)$.
	Notice that the $j$-th coordinate of $H_t(0,0)$ is equal to $[\dot \zeta(0)](h^j) = (\d h^j)(\bp) \, [\dot \zeta(0)] = 0$ by the properties of the MFCQ vector $[\dot \zeta(0)]$, for any $j = 1, \ldots, q$. 
	Thus we conclude $H_t(0,0) = 0$.

	The implicit function theorem ensures that there exists a function
	$y\colon(-\varepsilon_0,\varepsilon_0) \to \R^q$ of class $C^1$ such that
	$H(y(t),t) = 0$ and $y(0) = 0$ holds, and moreover, $\dot y(0) = H_y(0,0)^{-1}
	H_t(0,0) = 0$.

	Using $y(\cdot)$, we define, on a suitable open interval containing~$0$, the curve
	\begin{equation*}
		\gamma(t)
		\coloneqq
		\varphi^{-1} \varh(){ (\varphi \circ \zeta)(t) + (h \circ \varphi^{-1})'(x_0)\transp y(t) } \in \MM
		.
	\end{equation*}
	This curve is of class $C^1$ by construction, it satisfies $\gamma(0) = \varphi^{-1} (x_0 + 0) = \bp$ and generates the same tangent vector as the original curve $\zeta$.
	To see the latter, we consider an arbitrary $C^1$-function $f$ defined near~$\bp$ and calculate
	\begin{equation*}
		\begin{aligned}
			(f \circ \gamma)'(t)
			&
			=
			(f \circ \varphi^{-1})' \varh(){ (\varphi \circ \zeta)(t) + (h \circ \varphi^{-1})'(x_0)\transp y(t) }
			\\
			&
			\quad
			{} \cdot \bigh[]{ (\varphi \circ \zeta)'(t) + (h \circ \varphi^{-1})'(x_0)\transp \dot y(t) }
			.
		\end{aligned}
	\end{equation*}
	This implies
	\begin{equation*}
		[\dot \gamma(0)](f)
		=
		(f \circ \gamma)'(0)
		=
		(f \circ \varphi^{-1})'(x_0) \, (\varphi \circ \zeta)'(0) = (f \circ \zeta)'(0)
		=
		[\dot \zeta(0)](f)
		.
	\end{equation*}

	By construction, we have
	\begin{equation*}
		h(\gamma(t))
		=
		(h \circ \varphi^{-1}) \varh(){ (\varphi \circ \zeta)(t) + (h \circ \varphi^{-1})'(x_0)\transp y(t) }
		=
		H(y(t),t)
		=
		0
	\end{equation*}
	on a suitable interval $(-\varepsilon,\varepsilon)$.
	It remains to verify the conditions pertaining to the inequality constraints.
	When $i \in \II(\bp)$, then by continuity, $g^i(\gamma(t)) < 0$ for all $t \in (-\varepsilon_i, \varepsilon_i)$.
	When $i \in \AA(\bp)$, consider the auxiliary function $\phi(t) \coloneqq g^i(\gamma(t))$, which satisfies $\phi(0) = g^i(\gamma(0)) = 0$ and $\dot \phi(0) = (\d g^i)(\bp) [\dot \gamma(0)] = (\d g^i)(\bp) [\dot \zeta(0)] < 0$.
	An applications of Taylor's theorem now implies that there exists $\varepsilon_i > 0$ such that $\phi(t) < 0$ holds for $t \in (0,\varepsilon_i)$.
	Taking $\varepsilon = \min \{ \varepsilon_i: i = 1, \ldots, m\}$ finishes the proof.
\end{proof}

\begin{proposition}
	\label{proposition:MFCQ_implies_ACQ_M}
	MFCQ implies ACQ.
\end{proposition}
\begin{proof}
	In view of \cref{lemma:relation_between_cones_M}, we only need to show $\tangentconeM{\Omega}{\bp} \supset \linearizingconeM{\Omega}{\bp}$.
	To this end, suppose that $[\dot \gamma_0(0)]$ is an element of $\linearizingconeM{\Omega}{\bp}$ defined in \eqref{eq:linearizing_cone_M}, generated by some $C^1$-curve about~$\bp = \gamma_0(0)$.
	Moreover, let $\gamma$ be another $C^1$-curve about~$\bp$ such that $[\dot \gamma(0)]$ is an MFCQ vector, see \eqref{eq:MFCQ_conditions_M}.
	Finally, choose an arbitrary chart $\varphi$ about~$\bp$.

	For any $\tau \in (0,1]$, consider the curve
	\begin{equation*}
		\gamma_0 \oplus_{\varphi} (\tau \odot \gamma): t \mapsto \varphi^{-1} \bigh(){ (\varphi \circ \gamma_0)(t) + (\varphi \circ \gamma)(\tau \, t) - \varphi(\bp) } \in \MM,
	\end{equation*}
	which is defined on an interval $(-\varepsilon,\varepsilon)$ where both $\gamma$ and $\gamma_0$ are defined.
	Moreover by reducing $\varepsilon$ if necessary we achieve that $\gamma(t)$ and $\gamma(\tau \, t)$ belong to the domain of the chosen chart $\varphi$ and that $(\varphi \circ \gamma_0)(t) + (\varphi \circ \gamma)(\tau \, t) - \varphi(\bp)$ belongs to the image of $\varphi$ so that $\gamma_0 \oplus_\varphi (\tau \odot \gamma)$ is well-defined for $t \in (-\varepsilon,\varepsilon)$.

	We first show that $[\textstyle \frac{\d}{\dt} (\gamma_0 \oplus_{\varphi} (\tau \odot \gamma))(0)] \to [\dot \gamma_0(0)]$ as $\tau \searrow 0$.
	Indeed, for any $C^1$-function $f$ defined near~$\bp$, we have
	\begingroup
	\allowdisplaybreaks[4]
	\begin{align*}
		\MoveEqLeft
		(\d f)(\bp) [\textstyle \frac{\d}{\dt} (\gamma_0 \oplus_{\varphi} (\tau \odot \gamma))(0)]
		\\
		&
		=
		[\textstyle \frac{\d}{\dt} (\gamma_0 \oplus_{\varphi} (\tau \odot \gamma))(0)](f)
		&
		&
		\text{by definition of $(\d f)(\bp)$, see \eqref{eq:differential_of_a_real-valued_function}}
		\\
		&
		=
		\restr{\frac{\d}{\dt} \bigh[]{ f \circ (\gamma_0 \oplus_{\varphi} (\tau \odot \gamma)) }}{t=0}
		&
		&
		\text{by def.\ of tangent vectors, see \eqref{eq:tangent_vector_generated_by_a_curve}}
		\\
		&
		=
		\mrep{(f \circ \varphi^{-1})'(\varphi(\bp)) \varh[]{ \restr{\frac{\d}{\dt} \bigh(){ (\varphi \circ \gamma_0) + \tau \, (\varphi \circ \gamma) }}{t=0} } \quad \text{by the chain rule}}{}
		\\
		&
		=
		\restr{\frac{\d}{\dt} (f \circ \gamma_0)}{t=0}
		+
		\tau \, \restr{\frac{\d}{\dt} (f \circ \gamma)}{t=0}
		&
		&
		\text{by the chain rule}
		\\
		&
		=
		(\d f)(\bp) [\dot \gamma_0(0)]
		+
		\tau \, (\d f)(\bp) [\dot \gamma(0)]
		,
	\end{align*}
	\endgroup
	and the right-hand side converges to $[\dot \gamma_0(0)](f)$ as $\tau \searrow 0$.

	Next we show that the tangent vector along $\gamma_0 \oplus_{\varphi} (\tau \odot \gamma)$ is an MFCQ vector for any $\tau \in (0,1]$.
	Similarly as above, we have
	\begin{equation*}
		(\d g^i)(\bp) [\textstyle \frac{\d}{\dt} (\gamma_0 \oplus_{\varphi} (\tau \odot \gamma))(0)]
		=
		(\d g^i)(\bp) [\dot \gamma_0(0)]
		+
		\tau \, (\d g^i)(\bp) [\dot \gamma(0)]
	\end{equation*}
	which is negative for any $i \in \AA(\bp)$ since $\tau > 0$.
	Analogously, $(\d h^j)(\bp) [\textstyle \frac{\d}{\dt} (\gamma_0 \oplus_{\varphi} (\tau \odot \gamma))(0)] = 0$ follows for all $j = 1, \ldots, q$.
	This confirms that $\gamma_0 \oplus_{\varphi} (\tau \odot \gamma)$ is indeed an MFCQ vector.

	Fix $\tau \in (0,1]$.
	While $\gamma_0 \oplus_{\varphi} (\tau \odot \gamma)$ itself may not be feasible near $t = 0$, \cref{proposition:Ljusternik} shows that we can replace it by an equivalent $C^1$-curve which is feasible for $t \in [0,\varepsilon_\tau)$.
	In other words, the equivalence class $[\textstyle \frac{\d}{\dt} (\gamma_0 \oplus_{\varphi} (\tau \odot \gamma))(0)]$ belongs to the tangent cone $\tangentconeM{\Omega}{\bp}$.
	We showed above that $[\textstyle \frac{\d}{\dt} (\gamma_0 \oplus_{\varphi} (\tau \odot \gamma))(0)] \to [\dot \gamma_0(0)]$ holds as $\tau \searrow 0$.
	Since $\tangentconeM{\Omega}{\bp}$ is closed, we conclude that $[\dot \gamma_0(0)] \in \tangentconeM{\Omega}{\bp}$ holds. 
\end{proof}

Finally, the fact that ACQ implies GCQ is trivial, so \eqref{eq:chain_of_implications} is proved.

\section{Constraint Qualifications and the Polyhedron of Lagrange Multipliers}
\label{sec:CQs_and_Lagrange_multipliers}

In this section we consider a number of results relating various constraint
qualifications to the set of KKT multipliers at a local minimizer of
\eqref{eq:nlp}. To this end, we fix an arbitrary feasible point $\bp \in \Omega$
and consider the cone
\begin{equation}
	\label{eq:cone_of_}
	\FF(\bp)
	\coloneqq
	\{ f \in C^1(\MM,\R): \bp \text{ is a local minimizer for \eqref{eq:nlp}} \}
\end{equation}
of objective functions of class $C^1$ attaining a local minimum at $\bp$.
For $f \in \FF(\bp)$, we denote by
\begin{equation}
	\label{eq:polyhedron_of_Lagrange_multipliers}
	\Lambda(f;\bp)
	\coloneqq
	\{ (\mu,\lambda) \in \R_m \times \R_p: \eqref{eq:KKT_conditions_M} \text{ holds} \}
\end{equation}
the corresponding set of Lagrange multipliers.
It is easy to see that $\Lambda(f;\bp)$ is a closed, convex (potentially empty) polyhedron.

The following theorem is known in the case $\MM = \R^n$; see \cite{Gauvin1977,GouldTolle1971}
and \cite[Thms.~1~and~2]{Wachsmuth2012:1}.
It continues to hold verbatim for \eqref{eq:nlp}.
\begin{theorem}[Connections between CQs and Lagrange Multipliers]
	\label{theorem:Lagrange_multipliers_and_CQs}
	Suppose that $\bp \in \Omega$.
	\begin{enumerate}[label=$(\alph*)$,leftmargin=*]
		\item
			\label{item:GCQ_is_equivalent_to_nonempty_set_of_multipliers}
			The set $\Lambda(f;\bp)$ is non-empty for all $f \in \FF(\bp)$ if and only if \hyperref[item:GCQ]{(GCQ)} holds at $\bp$.
		\item
			\label{item:MFCQ_implies_compactness_of_set_of_multipliers}
			Suppose \hyperref[item:MFCQ]{(MFCQ)} holds at $\bp$.
			Then the set $\Lambda(f;\bp)$ is compact for all $f \in \FF(\bp)$.
		\item
			\label{item:compactness_of_set_of_multipliers_implies_MFCQ}
			If $\Lambda(f;\bp)$ is non-empty, compact for some $f \in \FF(\bp)$, then \hyperref[item:MFCQ]{(MFCQ)} holds at $\bp$.
		\item
			\label{item:LICQ_is_equivalent_to_singleton_set_of_multipliers}
			The set $\Lambda(f;\bp)$ is a singleton for all $f \in \FF(\bp)$ if and only if \hyperref[item:LICQ]{(LICQ)} holds at $\bp$.
	\end{enumerate}
\end{theorem}

\begin{proof}
	\cref{item:GCQ_is_equivalent_to_nonempty_set_of_multipliers}:
	\cref{theorem:first-order_necessary_optimality_conditions_M} shows that \hyperref[item:GCQ]{(GCQ)} implies $\Lambda(f;\bp) \neq \emptyset$ for any $f \in \FF(\bp)$.
	The converse is proved in \cite[Sec.~4]{GouldTolle1971} for the case $\MM = \R^n$; see also \cite[Thm.~6.3.2]{BazaraaShetty1976}.
	In order to utilize this result directly and to avoid stating an analogous one on $\MM$, we temporarily depart from our standing principle of minimizing the use of charts.
	Suppose that $(\d s)(\bp) \in \tangentconeM{\Omega}{\bp}^\circ \subseteq \cotangentspaceM{\bp}^{\circ}$ holds.
	Fix an arbitrary chart $(U,\varphi)$ about $\bp$.
	Suppose that $d$ is an arbitrary element from the tangent cone $\tangentconeRn{\varphi(U \cap \Omega)}{\varphi(\bp)}$.
	Then we can construct, as in the proof of \cref{lemma:comparison_of_tangent_cone_on_M_with_tangent_cone_through_chart}, the curve $\gamma(t) \coloneqq \varphi^{-1}(\varphi(\bp) + t \, d)$ so that $[\dot \gamma(0)] \in \tangentconeM{\Omega}{\bp}$ and $d = \bigh(){(\d \varphi)(\bp)} [\dot \gamma(0)]$ holds.
		We obtain
		\begin{equation*}
			0
			\ge
			(\d s)(\bp) \, [\dot \gamma(0)]
			=
			\restr{\frac{\d}{\dt} (s \circ \varphi^{-1} \circ \varphi \circ \gamma)}{t=0}
			=
			(s \circ \varphi^{-1})'(\varphi(\bp)) \, d
		\end{equation*}
		from $\tangentconeM{\Omega}{\bp}^\circ$, \cref{definition:differential}, the chain rule and the definition of $\gamma$.
	This shows $(s \circ \varphi^{-1})'(\varphi(\bp)) \in \tangentconeRn{\varphi(U \cap \Omega)}{\varphi(\bp)}^\circ$.

	Using \cite[Thm.~6.3.2]{BazaraaShetty1976} we can construct a $C^1$-function $r\colon\R^n \to \R$ such that $r'(\varphi(\bp)) = - (s \circ \varphi^{-1})'(\varphi(\bp))$ holds and $\varphi(\bp)$ is a local minimizer of \eqref{eq:nlp_in_charts} but with the objective $r$ in place of $(f \circ \varphi^{-1})$.
	By \cref{lemma:local_optimality_in_charts}, $\bp$ is a local minimizer of \eqref{eq:nlp} with objective $r \circ \varphi$.
	By assumption, $\Lambda(r \circ \varphi,\bp)$ is non-empty, i.e., there exist Lagrange multipliers $\mu$ and $\lambda$ such that
	\begin{equation*}
		(\d (r \circ \varphi))(\bp) + \mu \, (\d g)(\bp) + \lambda \, (\d h)(\bp) = 0
	\end{equation*}
	and \eqref{eq:KKT_conditions_2_M}, \eqref{eq:KKT_conditions_3_M} hold.
	In other words, $- (\d (r \circ \varphi))(\bp) \in \linearizingconeM{\Omega}{\bp}^\circ$, see \eqref{eq:negative_derivative_belongs_to_polar_of_linearizing_cone_M}.
	Moreover, the differentials of $r \circ \varphi$ and $-s$ at $\bp$ coincide since
	\begingroup
	\allowdisplaybreaks[4]
	\begin{align*}
		\MoveEqLeft
		(\d (r \circ \varphi))(\bp) \, [\dot \gamma(0)]
		\\
		&
		=
		[\dot \gamma(0)](r \circ \varphi)
		&
		&
		\text{by definition \eqref{eq:differential_of_a_real-valued_function} of the differential}
		\\
		&
		=
		\restr{\frac{\d}{\dt} (r \circ \varphi \circ \gamma)(t)}{t=0}
		&
		&
		\text{by definition \eqref{eq:tangent_vector_generated_by_a_curve} of a tangent vector}
		\\
		&
		=
		r'(x_0) \restr{\frac{\d}{\dt} (\varphi \circ \gamma)(t)}{t=0}
		&
		&
		\text{by the chain rule}
		\\
		&
		=
		- (s \circ \varphi^{-1})'(x_0) \restr{\frac{\d}{\dt} (\varphi \circ \gamma)(t)}{t=0}
		&
		&
		\text{by construction of $r$}
		\\
		&
		=
		- \restr{\frac{\d}{\dt}(s \circ \gamma)(t)}{t=0}
		&
		&
		\text{by the chain rule}
		\\
		&
		=
		- (\d s)(\bp) \, [\dot \gamma(0)]
		&
		&
		\text{by \eqref{eq:tangent_vector_generated_by_a_curve}, \eqref{eq:differential_of_a_real-valued_function}}
	\end{align*}
	\endgroup
	holds for arbitrary tangent vectors $[\dot \gamma(0)]$ in $\tangentspaceM{\bp}$.
	This shows that $\tangentconeM{\Omega}{\bp}^\circ \subseteq$ $\linearizingconeM{\Omega}{\bp}^\circ$ holds, i.e., the \hyperref[item:GCQ]{(GCQ)} is satisfied.

	\cref{item:MFCQ_implies_compactness_of_set_of_multipliers} and \cref{item:compactness_of_set_of_multipliers_implies_MFCQ}:
	a possible proof of these results is based on linear programming arguments in the Lagrange multiplier space and thus it is directly applicable here as well.
	We sketch the proof following \cite{Burke2014:1} for the reader's convenience.
	One first observes that the existence of an MFCQ vector in \hyperref[item:MFCQ]{(MFCQ)} is equivalent to the feasibility of the linear program
	\begin{equation}
		\label{eq:MFCQ_as_an_LP_feasibility_problem}
		\left\{
			\quad
			\begin{aligned}
				\text{Minimize} \quad & 0 \text{ w.r.t.\ } [\dot \gamma(0)] \in \tangentspaceM{\bp},
				\\
				\text{s.t.} \quad 
				&
				(\d g^i)(\bp) [\dot \gamma(0)]
				\le
				-1
				&
				&
				\text{for all } i \in \AA(\bp)
				,
				\\
				\text{and} \quad 
				&
				(\d h^j)(\bp) [\dot \gamma(0)]
				=
				0
				&
				&
				\text{for all } j = 1, \ldots, q
				.
			\end{aligned}
		\right.
	\end{equation}
	Using strong duality, one shows that \hyperref[item:MFCQ]{(MFCQ)} holds if and only if $\{ (\d h^j)(\bp) \}_{j=1}^q$ is linearly independent and 
	\begin{equation}
		\label{eq:MFCQ_as_positive_linear_independence}
		\begin{aligned}
			&
			\mu \, (\d g)(\bp) + \lambda \, (\d h)(\bp) = 0,
			\\
			&
			\mu_i \ge 0 \quad \text{for all } i \in \AA(\bp),
			\\
			&
			\mu_i = 0 \quad \text{for all } i \in \II(\bp),
		\end{aligned}
	\end{equation}
	has the only solution $(\mu,\lambda) = 0$.

	Now if $f \in \FF(\bp)$ holds and $\Lambda(f;\bp)$ is not bounded, then there exists a 
	sequence of Lagrange multipliers $(\mu^{(k)},\lambda^{(k)})$ whose Euclidan norm $\abs{(\mu^{(k)},\lambda^{(k)})}_2$ diverges to $\infty$. 
	Consequently, there exists a subsequence (which we do not re-label) such that $(\mu^{(k)},\lambda^{(k)})/\abs{(\mu^{(k)},\lambda^{(k)})}_2$ converges to some $(\mu,\lambda) \neq 0$.
	Exploiting the KKT conditions \eqref{eq:KKT_conditions_M} for $(\mu^{(k)},\lambda^{(k)})$ it follows that \eqref{eq:MFCQ_as_positive_linear_independence} holds.
	Consequently,
	\hyperref[item:MFCQ]{(MFCQ)} is violated.
	This shows \cref{item:MFCQ_implies_compactness_of_set_of_multipliers}.

	Conversely, if \hyperref[item:MFCQ]{(MFCQ)} does not hold, then there exists a non-zero vector $(\mu,\lambda)$ satisfying \eqref{eq:MFCQ_as_positive_linear_independence}.
	When $(\mu_0,\lambda_0) \in \Lambda(f;\bp)$, then $(\mu_0,\lambda_0) + t \, (\mu,\lambda)$ belongs to $\Lambda(f;\bp)$ as well for any $t \ge 0$, hence $\Lambda(f;\bp)$ is not compact.
	This confirms \cref{item:compactness_of_set_of_multipliers_implies_MFCQ}.

	\cref{item:LICQ_is_equivalent_to_singleton_set_of_multipliers}:
	We have proved in \cref{sec:KKT_and_CQs} that \hyperref[item:LICQ]{(LICQ)} implies \hyperref[item:GCQ]{(GCQ)}, so $\Lambda(f;\bp)$ is non-empty.
	The uniqueness of the Lagrange multipliers then follows immediately from \eqref{eq:KKT_conditions_1_M}.
	The converse statement is proved in \cite[Thm.~2]{Wachsmuth2012:1}, which applies without changes.
\end{proof}

\section{Numerical Example}
\label{sec:Example}

In this section we present a numerical example in which the fulfillment of the KKT conditions \eqref{eq:KKT_conditions_M} is used as an algorithmic stopping criterion.
While the framework of a smooth manifold was sufficient for the discussion of first-order optimality conditions, we require more structure for algorithmic purposes.
Therefore we restrict the following discussion to complete Riemannian manifolds.
In this section we denote tangent vectors by the symbol $\xi$ instead of $[\dot \gamma(0)]$.

A manifold is Riemannian if its tangent spaces are equipped with a smoothly varying metric $\dual{\cdot}{\cdot}_{\bp}$.
This allows the conversion of the differential of the objective~$f$, $(\d f)(\bp) \in \cotangentspaceM{\bp}$, to the gradient~$\nabla f(\bp) \in \tangentspaceM{\bp}$, which fulfills
\begin{align*}
	\dual{\xi}{\nabla f(\bp)}_{\bp},
	=
	(\d f)(\bp) \, \xi
	\quad
	\text{for all }
	\xi\in\tangentspaceM{\bp}.
\end{align*}
Completeness of a Riemannian manifold refers to the fact that geodesics emanating from any point $\bp \in \MM$ in the direction of an arbitrary tangent vector $\xi$ exist for all time $t \in \R$.

The Riemannian center of mass, also known as (Riemannian) mean was introduced in~\cite{Karcher1977} as a variational model.
Given a set of points $\bd_i$, $i=1,\ldots, N$, their Riemannian center is defined as the minimizer of
\begin{align}
	\label{eq:ConstrCM_objective_Riemannian_center_of_mass}
	f(\bp) \coloneqq \frac{1}{N}\sum_{i=1}^N d^2_{\MM}(\bp,\bd_i),
\end{align}
where $d_{\MM}\colon \MM\times\MM\to\R$ is the distance on the Riemannian
manifold $\MM$.

We extend this classical optimization problem on manifolds by adding the constraint that the minimizer should lie within a given ball of radius $r > 0$ and center $\bc \in \MM$.
We obtain the following constrained minimization problem of the form \eqref{eq:nlp},
\begin{equation}
	\label{eq:ConstrCM_nlpExample1}
	\left\{
		\quad
		\begin{aligned}
			\text{Minimize} \quad & f(\bp) \text{ w.r.t.\ } \bp \in \MM, \\
			\text{s.t.} \quad & d^2_{\MM}(\bp,\bc) - r^2 \le 0,
		\end{aligned}
	\right.
\end{equation}
with associated Lagrangian
\begin{align}
	\LL(\bp,\mu) = \frac{1}{N}\sum_{i=1}^N d^2_{\MM}(\bp,\bd_i) + \mu \, (d^2_{\MM}(\bp,\bc) - r^2).
\end{align}
It can be shown, see for example~\cite{Bacak2014,AfsariTronVidal2013}, that the
objective and the constraint are $C^1$-functions whose gradients are given by
the tangent vectors
\begin{equation}
	\label{eq:ConstrCM_gradients_of_objective_and_constraint}
	\nabla f(\bp)
	=
	- \frac{2}{N}\sum_{i=1}^N \log_{\bp}\bd_i
	\quad \text{and} \quad
	\nabla g(\bp)
	=
	-2 \log_{\bp}\bc
	.
\end{equation}
Here $\log$ denotes the logarithmic (or inverse exponential) map on $\MM$.
In other words, $\log_{\bp}\br \in \tangentspaceM{\bp}$ is the initial velocity of the geodesic curve starting in $\bp \in \MM$ which reaches $\br \in \MM$ at time~1.

In view of \eqref{eq:ConstrCM_gradients_of_objective_and_constraint}, the KKT conditions \eqref{eq:KKT_conditions_M} become
\begin{align*}
	&
	0
	=
	(\d\LL)(\bp,\mu)[\xi]
	=
	\frac{1}{N}\sum_{i=1}^N \dual{\xi}{-2\log_{\bp}\bd_i}_{\bp}
	+
	\mu \, \dual{\xi}{-2\log_{\bp}\bc}_{\bp}
	\quad
	\text{for all }\xi\in\tangentspaceM{\bp}
	\\
	&
	\mu
	\geq 0,
	\quad
	d^2_{\MM}(\bp,\bc)
	\leq r^2,
	\quad
	\mu \, (d^2_{\MM}(\bp,\bc) - r^2) = 0
	.
\end{align*}

\begin{figure}\centering
	\begin{subfigure}[t]{.49\textwidth}
		\includegraphics[width=.95\textwidth]{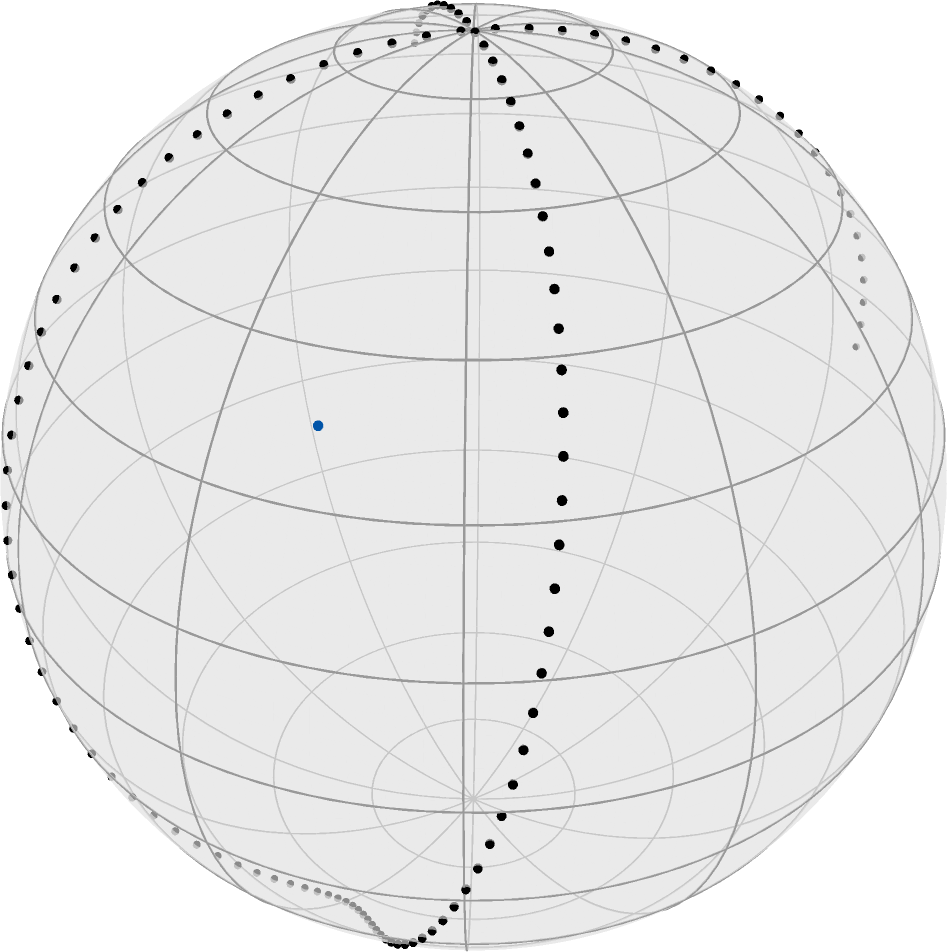}
		\caption{Data points $\bd_i$ and their mean $\bar \bp$, the (unconstrained) Riemannian center of mass.}
		\label{subfig:ConstrCM:S2:Data}
	\end{subfigure}
	\\
	\begin{subfigure}[t]{.49\textwidth}
		\includegraphics[width=.95\textwidth]{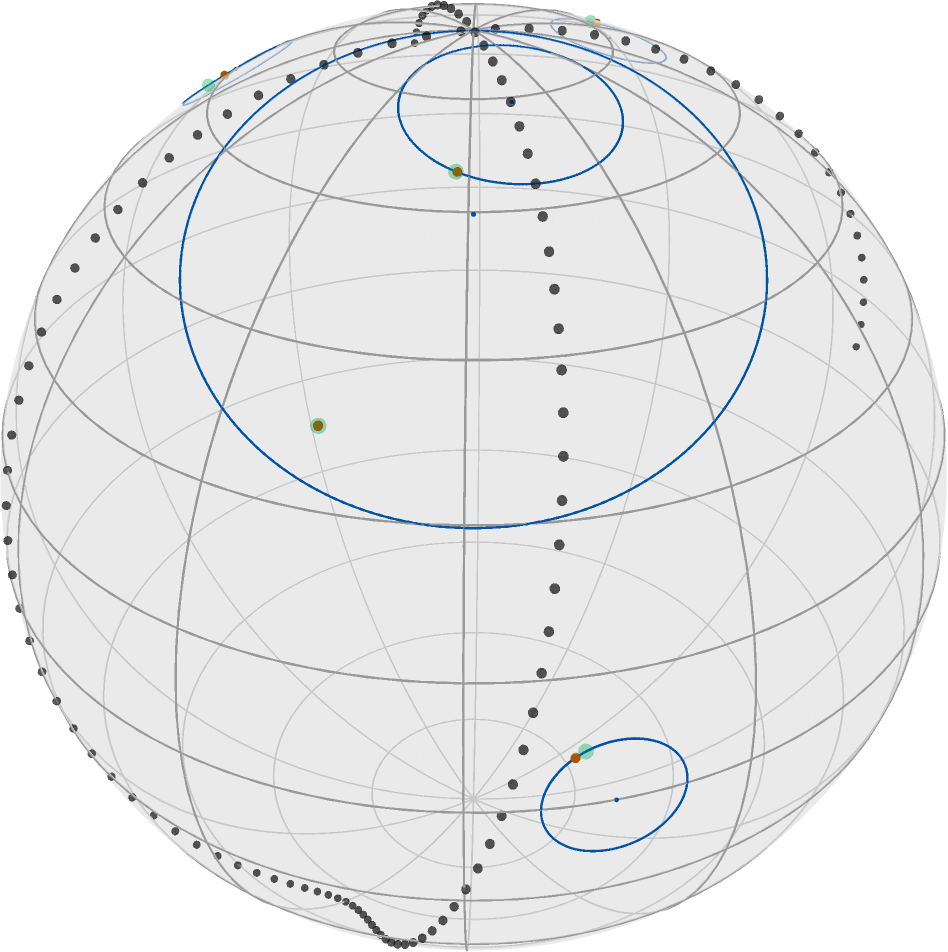}
		\caption{Constrained solutions of \eqref{eq:ConstrCM_nlpExample1} (light green) and projected unconstrained means $\proj_{\Omega}(\bar \bp)$ (orange)
		for five different feasible sets (blue).}
		\label{subfig:ConstrCM:S2:View1}
	\end{subfigure}
	\hfill
	\begin{subfigure}[t]{.49\textwidth}
		\includegraphics[width=.95\textwidth]{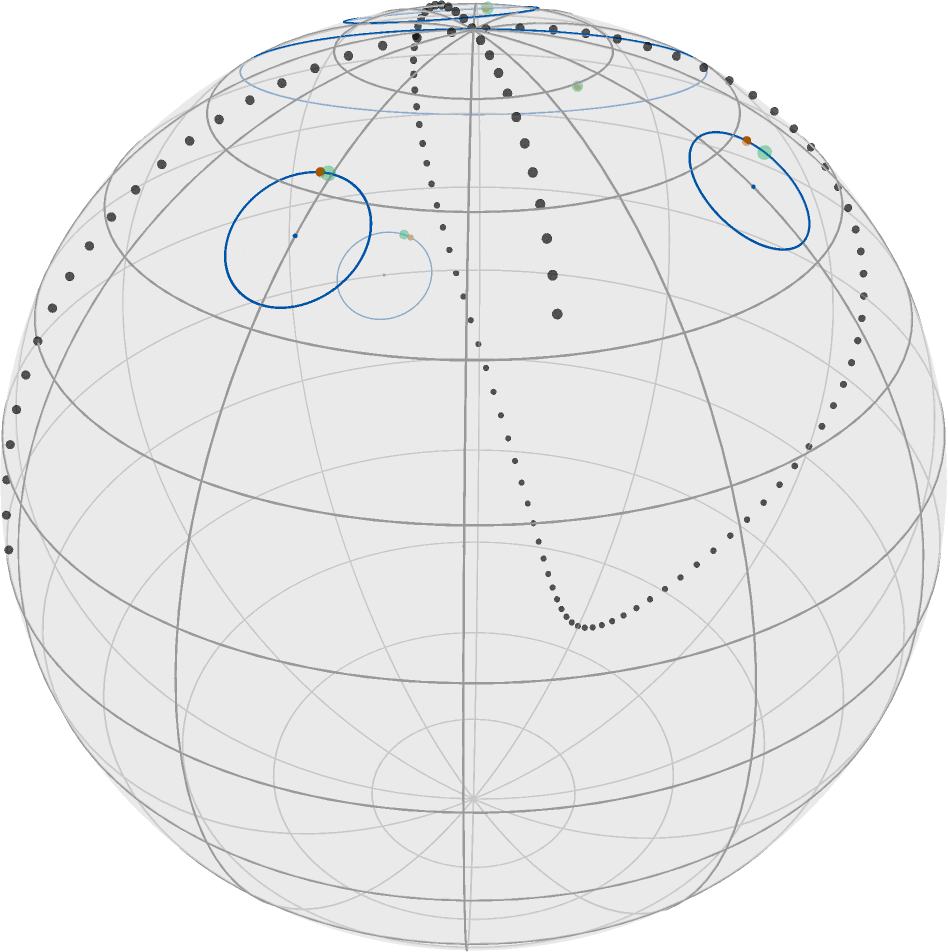}
		\caption{Same as \cref{subfig:ConstrCM:S2:View1}, rotated by 180 degrees.}
		\label{subfig:ConstrCM:S2:View2}
	\end{subfigure}
	\caption{Constrained centers of mass for five different feasible sets (centers and boundaries of the feasible sets shown in blue).
	Unlike in $\R^2$, the minimizers $\bp^*$ (light green) differ from the mean $\bar \bp$ projected onto the feasible set \eqref{eq:ConstrCM_projecting_the_unconstrained_mean} (orange).}
	\label{fig:ConstrCM:S2}
\end{figure}
In our example we choose $\MM = \mathbb S^2 \coloneqq \{\bp \in \R^3 : \lvert\bp\rvert_2 = 1 \}$ the two-dimensional manifold of unit vectors in $\R^3$ or $2$-sphere. We further have to restrict the data to not include antipodal points, i.e.~the case that for some $i,j\in\{1,\ldots,N\}$ it holds $\bd_i=-\bd_j$ is excluded.
The Riemannian metric is inherited from the ambient space $\R^3$.
Since the feasible set
\begin{equation}
	\label{eq:ConstrCM_feasible_set}
	\Omega
	\coloneqq
	\{ \bp \in \mathbb S^2: d_{\MM}(\bp,\bc) \le r \}
\end{equation}
is compact, a global minimizer to \eqref{eq:ConstrCM_nlpExample1} exists. Notice, however, that unlike in the flat space $\R^2$, minimizers are not necessarily unique.

Even in the absence of convexity, the LICQ is satisfied at every solution $\bp^*$ unless $\bp^* = \bc$ holds, which is equivalent to the unconstrained mean $\bar \bp$ coinciding with the center $\bc$ of the feasible set.
This does not happen for the data we use.
Consequently, the Lagrange multiplier is unique by \cref{theorem:Lagrange_multipliers_and_CQs}.

In our example, we choose a set of $N = 120$ data points $\bd_i$ as shown in \cref{subfig:ConstrCM:S2:Data}.
Their unconstrained Riemannian center of mass $\bar \bp$ is shown in blue.
We solve five variants of problem \eqref{eq:ConstrCM_nlpExample1} which differ w.r.t.\ the centers $\bc_i$ and radii $r_i$ of the feasible sets $\Omega_i$.
The boundaries $\partial \Omega_i$ of the feasible sets, which are spherical caps,  are displayed in blue in \cref{subfig:ConstrCM:S2:View1} (front view) and \cref{subfig:ConstrCM:S2:View2} (back view).
The constrained solutions $\bp^*_i$ are shown in light green in \cref{subfig:ConstrCM:S2:View1,subfig:ConstrCM:S2:View2}.

Each instance of problem \eqref{eq:ConstrCM_nlpExample1} is solved using
a projected gradient descent method.
Since it is a rather straightforward generalization of an unconstrained
gradient algorithm, see for instance~\cite[Ch.~4, Alg.~1]{AbsilMahonySepulchre2008}, we only briefly sketch it here.
We utilize the fact that the feasible set $\Omega$ is closed and geodesically
convex when $r<\pi/2$, i.e., for any two points $\bp,\bq\in\Omega$,
all (shortest) geodesics
connecting these two points lie inside $\Omega$.
In this case the projection~$\proj_\Omega\colon\MM\to\Omega$ onto $\Omega$ is defined by
\begin{align*}
	\proj_\Omega(\bp)\coloneqq\argmin_{\bq\in\Omega} d_\MM(\bp,\bq).
\end{align*}
It can be computed in closed form, namely
\begin{align}
	\label{eq:ConstrCM_projecting_the_unconstrained_mean}
	\proj_{\Omega}(\bp) =
	\exp_{\bc}\bigl(b\log_{\bc}\bp\bigr),
	\quad \text{where }
	b = \min\Bigl\{\frac{r}{d_{\MM}(\bp,\bc)},1\Bigr\}
	,
\end{align}
whenever the logarithmic map is uniquely determined.
This in turn holds whenever $\bp\neq-\bc$.
\begin{algorithm}[tbp]
	\caption{Projected gradient descent algorithm}
	\label{alg:projGradDesc}
	\begin{algorithmic}
		\Require{%
				an objective function $f\colon\MM\to\R$; a closed and convex set $\Omega$;
				a fixed step size~$s>0$;
				and an initial value $\bp^{(0)}\in\MM$
			}
		\State $k \gets 0$
		\Repeat
		\State $ \bp^{(k+1)} \gets \proj_\Omega\bigl( \exp_{\bp^{(k)}}(-s \nabla f(\bp^{(k)}))\bigr)$
		\State $k \gets k+1$
		\Until a convergence criterion is reached
		\State\Return $\bp^* = \bp^{(k)}$
	\end{algorithmic}
\end{algorithm}

The projected gradient descent algorithm is given as pseudo code in \cref{alg:projGradDesc}.
The unconstrained problem with solution $\bar \bp$ is solved similarly, omitting the projection step.
This amounts to the classical gradient descent method on manifolds as given in \cite[Ch.~4, Alg.~1]{AbsilMahonySepulchre2008}.
In our experiments we set the step size to $s=\frac{1}{2}$
and used the
first data point as initial data $\bp^{(0)}=\bd_1$,
which is the \eqq{bottom left} data point in~\cref{subfig:ConstrCM:S2:View2},
to solve the constrained instances.
The algorithm was implemented within the Manifold-valued Image Restauration
Toolbox (MVIRT)%
\footnote{available open source
at~\url{http://ronnybergmann.net/mvirt/}.}
\cite{Bergmann2017}, providing a direct access to the necessary functions for
the manifold of interest and the required algorithms.

Notice that in $\R^2$, the constrained mean of a set of points can simply be
obtained by projecting the unconstrained mean $\bar \bp$ onto the feasible disk.
In $\mathbb S^2$, this would amount to $\proj_{\Omega}(\bar \bp)$, but this differs, in general, from the solution of \eqref{eq:ConstrCM_nlpExample1} due to the curvature of $\mathbb S^2$.
For comparison, we show the result of $\proj_{\Omega}(\bar \bp)$ in orange in \cref{subfig:ConstrCM:S2:View1,subfig:ConstrCM:S2:View2}.

By design, gradient type methods do not utilize Lagrange multiplier estimates.
At an iterate $\bp^{(k)}$, we therefore estimate the Lagrange multiplier
$\mu^{(k)}$ by a least squares approach, which amounts to
\begin{equation}
	\label{eq:ConstrCM_Lagrange_multiplier_estimate}
	\mu^{(k)}
	\coloneqq
	- \frac{\dual{\nabla g(\bp^{(k)})}{\nabla f(\bp^{(k)})}_{\bp^{(k)}}}{\dual{\nabla g(\bp^{(k)})}{\nabla g(\bp^{(k)})}_{\bp^{(k)}}}
	.
\end{equation}
We then evaluate the gradient of the Lagrangian,
\begin{equation}
	\label{eq:ConstrCM_gradient_of_Lagrangian}
	\nabla_{\bp} \LL(\bp^{(k)},\mu^{(k)})
	=
	-\frac{2}{N}\sum_{i=1}^N \log_{\bp^{(k)}} \bd_i
	-2 \, \mu^{(k)} \log_{\bp^{(k)}} \bc ,
\end{equation}
and utilize its norm squared $n^{(k)} \coloneqq \dual{\nabla_{\bp} \LL(\bp^{(k)},\mu^{(k)})}{\nabla_{\bp} \LL(\bp^{(k)},\mu^{(k)})}_{\bp^{(k)}}$ as a stopping criterion.

For two of the five test cases we display the iteration history in \cref{tab:iterates}.
The first example is the large circle with center $\bc_1 \approx (0.4319,0.2592,0.8639)\transp$ and radius $r_1=\frac{\pi}{6}$.
For this setup the constraint is inactive and $\bp^*_1 = \bar \bp$ holds.
The second example is shown to the right of~\cref{subfig:ConstrCM:S2:View2}, and it is given by~$\bc_2 \approx (0,-0.5735,0.8192)\transp$ and $r_2 = \frac{\pi}{24}$.
For this and remaining three cases the constraint is active.

\begin{table}\centering
	\caption{Iteration history of \cref{alg:projGradDesc} for two instances of problem \eqref{eq:ConstrCM_nlpExample1}.}
	\label{tab:iterates}
	\sisetup{detect-family,detect-display-math = true }
	\begin{subtable}[t]{.545\textwidth}
		\subcaption*{Results for $(\bc_1,r_1)$.}
		\begin{tabular}{@{}S[table-format=2.0]S[table-format=1.3]*2{S[table-format=-1.2e+2]}@{}}\toprule
			{$k$} & {$f(\bp^{(k)})$} & {$n^{(k)}$} & {$\mu^{(k)}$}\\\midrule
			1 & 1.9129 & 0.6540 & 1.1722\\
			2 & 1.4172 & 0.1243 & 0.2755\\
			3 & 1.3754 & 0.0169 & -0.0847\\
			4 & 1.3695 & 0.0029 & -0.0811\\
			5 & 1.3684 & 0.0005 & -0.0403\\
			6 & 1.3682 & 0.0001 & -0.0180\\
			7 & 1.3682 & 1.18e-5 & -0.0078\\
			8 & 1.3682 & 3.26e-6 & -0.0034\\
			9 & 1.3682 & 6.02e-7 & -0.0014\\
			10 & 1.3682 & 1.11e-7 & -0.0006\\
			11 & 1.3682 & 2.05e-8 & -0.0003\\
			12 & 1.3682 & 3.79e-9 & -0.0001\\
			13 & 1.3682 & 6.99e-10 & -4.94e-5\\
			14 & 1.3682 & 1.29e-10 & -2.12e-5\\
			15 & 1.3682 & 2.38e-11 & -9.13e-6\\
			16 & 1.3682 & 4.40e-12 & -3.93e-6\\
			17 & 1.3682 & 8.13e-13 & -1.69e-6\\
			18 & 1.3682 & 1.50e-13 & -7.25e-7\\
			19 & 1.3682 & 2.77e-14 & -3.11e-7\\
			20 & 1.3682 & 5.12e-15 & -1.34e-7\\
			21 & 1.3682 & 9.45e-16 & -5.75e-8\\
			22 & 1.3682 & 1.74e-16 & -2.47e-8\\\bottomrule
		\end{tabular}
	\end{subtable}
	\begin{subtable}[t]{.445\textwidth}
		\subcaption*{Results for $(\bc_2,r_2)$.}
		\begin{tabular}{@{}S[table-format=1.0]S[table-format=1.4]S[table-format=1.2e+2]S[table-format=1.4]@{}}\toprule
			{$k$} & {$f(\bp^{(k)})$} & {$n^{(k)}$} & {$\mu^{(k)}$}\\\midrule
			1 & 2.2190 & 2.1771 & 1.3833\\
			2 & 2.0215 & 0.0011 & 1.2454\\
			3 & 2.0214 & 5.04e-6 & 1.2475\\
			4 & 2.0214 & 2.40e-8 & 1.2476\\
			5 & 2.0214 & 1.15e-10 & 1.2477\\
			6 & 2.0214 & 5.50e-12 & 1.2477\\
			7 & 2.0214 & 2.63e-15 & 1.2477\\
			8 & 2.0214 & 1.25e-17 & 1.2477\\\bottomrule
		\end{tabular}
	\end{subtable}
\end{table}

Since the unconstrained Riemannian mean is within the feasible set for the first
example of $(\bc_1,r_1)$, the projection is the identity after the first
iteration. Hence for this case, the (projected) gradient descent algorithm
computes the unconstrained mean similar to~\cite{AfsariTronVidal2013}.
We obtain $\bp^*_1 = \bar\bp =\proj_{\Omega_{1}}(\bar \bp)$.
Looking at the gradients~$\nabla f$ 
and $\nabla g$ we see, cf.~\cref{subfig:ConstrCM:S2:grads1},
that $\nabla f=0$ 
while the constraint function $g$ yields a gradient
pointing towards the boundary $\partial\Omega_{1}$ of the feasible set. Clearly, the optimal Lagrange multiplier is zero in this case.
The iterates (green points) follow a typical
gradient descent path of a Riemannian center of mass computation.
Notice that the Lagrange multiplier happens to approach zero from below
in this case.
While the objective decreases, the distance from $\bc_1$, and thus $g$
increases, leading to a negative multiplier estimate $\mu^{(k)}$.

For the second case, $(\bc_2,r_2)$ the unconstrained mean lies outside the
feasible set, and the constraint $g$ is strongly active, i.e., the multiplier $\mu$ is strictly positive.
As we mentioned earlier,  the optimal solution $\bp^*_{2}$ 
is different from $\proj_{\Omega_{2}}(\bar\bp)$, their distance
is \SI{0.0409}, which is due to the curvature of the manifold.

\begin{figure}\centering
	\begin{subfigure}[t]{.49\textwidth}
		\includegraphics[width=.9233\textwidth]{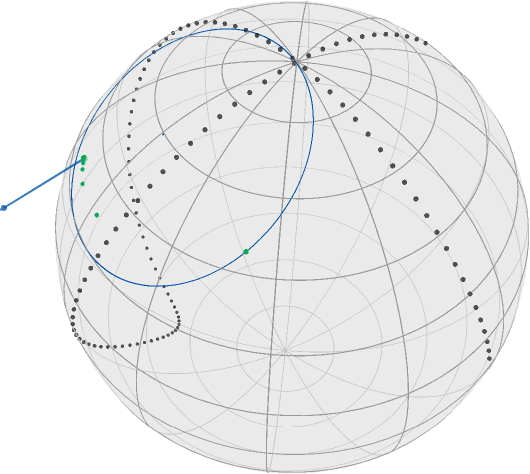}
		\caption{Constraint data $(\bc_1,r_1)$.}
		\label{subfig:ConstrCM:S2:grads1}
	\end{subfigure}
	\begin{subfigure}[t]{.49\textwidth}
		\includegraphics[width=\textwidth]{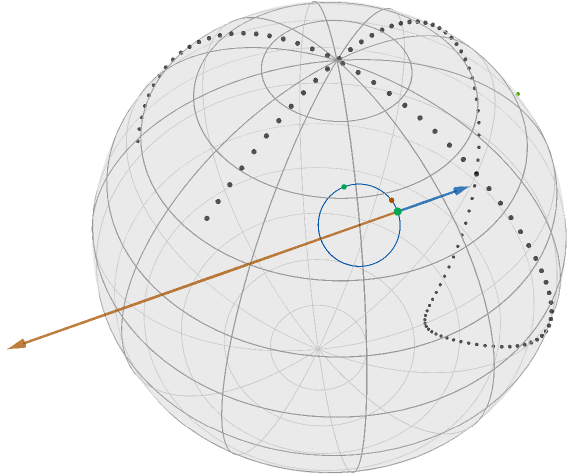}
		\caption{Constraint data $(\bc_2,r_2)$.}
		\label{subfig:ConstrCM:S2:grads2}
	\end{subfigure}
	\caption{Iterates (green) of the projected gradient method and the final gradients
	of the objective $f$ (orange) as well as the constraint $g$ (blue).}
\end{figure}

\section*{Acknowledgments}

We would like to thank two anonymous reviewers for their constructive comments which helped to significantly improve the manuscript.

\printbibliography
\end{document}